\documentclass[reqno]{amsart} 
\usepackage{latexsym,amssymb,amsmath,pdfsync,graphicx} 
\usepackage{enumerate} 
\usepackage{amsfonts}
\usepackage{eucal} 
\usepackage[bookmarks,colorlinks=true]{hyperref}

\usepackage[all]{xy}

\usepackage{ifthen, xspace}

\newcommand{\bsc}{
\usefont{T1}{cmr}{bx}{sc}} 
\newcommand{\pending}[1][]{{\bsc\ifthenelse{\equal{#1}{}}{[to be written]}{[to be written: {\rm #1}]}}\xspace} 
\newcommand{\prelims}[1][]{{\bsc\ifthenelse{\equal{#1}{}}{[add to prelims]}{[add to prelims: {\rm #1}]}}\xspace} 
\newcommand{\adding}[1][]{{\bsc\ifthenelse{\equal{#1}{}}{[to add?]}{[to add?: {\rm #1}]}}\xspace} 
\newcommand{\wondering}[1][]{{\bsc\ifthenelse{\equal{#1}{}}{[Maybe]}{[Maybe: {\rm #1}]}}\xspace}

\newcommand{\rem}[1]{\relax}

\newcommand{\Cyl}[1]{\ensuremath{[\![{#1}]\!]}} 
\newcommand{\seq}[1]{{\langle{#1}\rangle}} 
\newcommand{\uhr}[1]{ {\upharpoonright_{#1}}}

\newcommand \tth{{}^{\textup{th}}} \DeclareMathOperator \dom{dom}

\newcommand{\Tur}{\mathrm{T}}

\renewcommand{\hat}{\widehat} 
\renewcommand{\tilde}{\widetilde} 

\newcommand{\converge}{\!\!\downarrow} 
\newcommand{\diverge}{\!\!\uparrow} 
\newcommand {\PP}{\mathcal P} 
\newcommand{\CC}{ \mathcal C} 
\newcommand {\GG}{\mathcal G}
\newcommand {\RR}{\mathcal R} 
\newcommand {\WW}{\mathcal W}

\newcommand{\QQ}{{\mathcal{Q}}}

\newcommand{\Then}{\Rightarrow} 
\newcommand{\w}{\omega} 
\renewcommand{\epsilon}{\varepsilon} 
\newcommand{\s}{\sigma} 
 
\newcommand{\SJT}{\mbox{\textit{\textsf{SJT$_{\text{c.e.}}$}}}}
\newcommand{\Superlow}{\mbox{\textit{\textsf{Superlow}}}} 
\newcommand{\Superhigh}{\mbox{\textit{\textsf{Superhigh}}}} 
\newcommand{\wce}{\mbox{\textit{\textsf{$\w$-c.e.}}}} 
\renewcommand {\le}{\leqslant} 
\renewcommand {\ge}{\geqslant}

\theoremstyle{plain} 
\newtheorem{theorem}{Theorem}[section] 
 
\newtheorem{proposition}[theorem]{Proposition} 
\newtheorem{lemma}[theorem]{Lemma} 
\newtheorem{corollary}[theorem]{Corollary} 
 
\newtheorem{claim}[theorem]{Claim}
\newtheorem{claimin}{Claim}

\theoremstyle{definition} 
\newtheorem{definition}[theorem]{Definition} 

\theoremstyle{remark} 
\newtheorem{fact}[theorem]{Fact} 
\newtheorem{remark}[theorem]{Remark}

\newcommand{\nicebox}[1] {\vsp 
\noindent \fbox{\parbox{.98\linewidth}{ \centering{#1} }} }

\newcommand{\HHH}{\mathcal H}

\newcommand{\HH}{\mathcal{K}} 
\newcommand{\cost}{c_{\HH}}

\DeclareMathOperator{\BLR}{BLR}
\newcommand{\exo}[1]{\exists #1 \, } 
\newcommand{\fao}[1]{\forall #1 \, } 
\newcommand{\GL}[1]{\mbox{\rm GL}_{#1}}

\newcommand{\lwtt}{\le_{\mathrm{wtt}}} 
\newcommand{\ltt}{\le_{\mathrm{tt}}}

\newcommand{\lt}{\le_{\mathrm{T}}}

\newcommand{\CCC}{\mathcal{C}}

\newcommand{\SI}[1]{\Sigma^0_{#1}} 
\newcommand{\PI}[1]{\Pi^0_{#1}} 
\newcommand{\PPI}{\PI{1}}

\newcommand{\Kuc}{Ku{\v c}era} 
\newcommand{\ML}{Martin-L{\"o}f}

\newcommand{\Halt}{\ES'}

\newcommand{\NN}{\w} 

\newcommand{\SIII}{\Sigma^0_3} 
 
\newcommand{\DII}{\Delta^0_2} 
\newcommand{\nle}{\not \le}

\newcommand{\uhrn}{\upharpoonright}

\newcommand{\ex}{\exists} 
\newcommand{\fa}{\forall} 
\newcommand{\LR}{\Leftrightarrow} 
 
\newcommand{\RA}{\Rightarrow} 
 
\newcommand{\rra}{\ \rightarrow\ } 
\newcommand{\ria}{\rightarrow}

\newcommand{\n}{
\noindent}

\newcommand{\sub}{\subseteq}

\renewcommand{\land}{\&} 
 
\newcommand{\ES}{\emptyset}

\renewcommand{\hat}{\widehat} 
\renewcommand{\tilde}{\widetilde}

\newcommand{\lland}{\ \land \ }

\newcommand{\la}{\langle} 
\newcommand{\ra}{\rangle}

\newcommand{\bi}{
\begin{itemize}
	} 
	\newcommand{\ei}{
\end{itemize}
} 
\newcommand{\bc}{
\begin{center}
	} 
	\newcommand{\ec}{
\end{center}
}

\newcommand{\sss}{\sigma} 
\newcommand{\aaa}{\alpha}

\newcommand{\leT}{\le_\Tur}
\newcommand{\equivT}{\equiv_\Tur}

\newcommand{\leb}{\mathbf{\lambda}} 
\newcommand{\eps}{\epsilon} 
\newcommand{\Om}{\Omega} 
\newcommand{\twoset}{\{0,1\}}

\newcommand{\sN}[1]{_{#1\in \w}} 
\newcommand{\estring}{\varnothing}

\newcommand{\Opcl}[1]{[#1]^\prec} 
\newcommand{\tp}[1]{2^{#1}}

\newcommand{\SSS}{\mathcal{S}}

\newcommand{\BBB}{\mathcal{B}}

\newcommand{\vsp}{\vspace{6pt}} 
\newcommand{\vsps}{\vspace{3pt}}

\newcommand{
\use}{\mbox{\rm \textsf{use }}} 
\newcommand{\kuc}{\mbox{\rm \textsf{kuc}}}

\begin{document} 

\title[Characterizing SJT sets via randomness]{Characterizing the strongly jump-traceable sets via randomness}

\author{Noam Greenberg} \address{School of Mathematics, Statistics and Operations Research \\
Victoria University of Wellington \\
Wellington \\
New Zealand} \email{greenberg@msor.vuw.ac.nz}

\author{Denis R.~Hirschfeldt} \address{Department of Mathematics \\
University of Chicago \\
Chicago, IL \\
U.S.A.} \email{drh@math.uchicago.edu}

\author {Andr\'e~Nies} \address{Department of Computer Science \\
University of Auckland \\
Auckland \\
New Zealand} \email{andre@cs.auckland.ac.nz}

\keywords{computability, randomness, lowness, traceability}

\subjclass[2010]{Primary: 03D32; Secondary: 03D25, 03D30, 03D80, 03F60, 68Q30}

\thanks{The first and third authors were partially supported by the Marsden Fund of New Zealand.
%, grant no.\ 03-UOA-130.}
The second author was partially supported by the National Science Foundation of the United States, grants DMS-0801033 and DMS-0652521.}
\begin{abstract}
We show that if a set $A$ is computable from every superlow $1$-random set, then $A$ is strongly jump-traceable. Together with a result from \cite{GreenbergNies:benign}, this theorem shows that the computably enumerable (c.e.) strongly jump-traceable sets are exactly the c.e.\  sets computable from every superlow $1$-random set.

We also prove the analogous result for superhighness: a c.e.\ set is strongly jump-traceable if and only if it is computable from every superhigh $1$-random set. 

Finally, we show that for each cost function $c$ with the limit condition there is a $1$-random $\DII$ set $Y$ such that every c.e.\ set $A \leT Y$ obeys $c$.  To do so,  we connect cost function strength and the strength of   randomness notions. Together with a theorem from \cite{GreenbergNies:benign}, this result gives a  full correspondence between obedience of cost functions and being computable from $\DII$ $1$-random sets. 
\end{abstract}

%This theorem extends the result that there is a random $\DII$ set such that every c.e.~set $A\leT Y$ is strongly jump-traceable, and serves as a full converse of a result from \cite{GreenbergNies:benign}, thus giving full correspondence between obedience of cost functions and being computable from $\DII$ random sets. 

\maketitle

\section{Background and motivation}

There are two aspects to the information content of sets of natural numbers. In terms of computational complexity, a set of numbers is considered to code a lot of information if it is useful as an oracle for relative computation. In terms of effective randomness, difficulty to detect patterns in the set marks it as complicated, or random. The interaction between these two aspects of complexity is the focus of much current research in computability theory. 

Although earlier research naturally gravitated toward the  complex, recent findings have shown rich structure in the region of the simple.  Properties of sets that indicate    being  uncomplicated are called \emph{lowness properties}. They have proved to be essential in the understanding of random sets, and of the connections between computability and randomness along the entire spectrum of complexity. 

A lowness property that is central to this study is that of $K$-triviality. A series of results by Downey, Hirschfeldt, Nies, and Stephan (see \cite{HirschfeldtNiesStephan:UsingRandomOracles, Nies:_lowness_and_randomness}) developed penetrating techniques for the study of several classes of low sets. These results have established the coincidence of several such notions, three of the important ones being: $K$-triviality (being far from random); lowness for randomness (not being able to detect new patterns in random sets); and being computable from a relatively random oracle. (Here, ``random'' means Martin-L\"of random, or $1$-random, as defined in Subsection \ref{subsec:randomndess}.) This coincidence established the robustness of this class. Further results have demonstrated its usefulness and importance to the field; see, for example, \cite{DHNT}. 

The diverse characterizations of the $K$-trivial sets, and the techniques used to study them, have led to three  paradigms  for understanding lowness of a set $A$ of natural numbers, introduced by Nies \cite{Nbook,Nies:tut}:

\n {\bf 1.} \emph{Being weak as an oracle.} This paradigm  means that  $A$ is not very useful as an oracle for Turing machines. This is the oldest way of thinking about lowness. For instance, $A$ is of hyper-immune free degree if it does not compute fast growing functions: each function computed by $A$ is dominated by a computable function.  Some formal instances of the paradigm are   expressed through $A'$,  the halting set relative to $A$. For instance,  the traditional notion, simply called ``low'', states that $A'$ is as simple as possible in the   Turing degrees. The newer notion of superlowness states that $A'$ is as simple as possible in the truth-table degrees.  

\n {\bf 2.} \emph{Being computed by many oracles.} Traditionally, there were no interesting answers to the question ``how many sets compute $A$?''; the answer is always ``uncountably many''---indeed continuum many---but unless $A$ is computable (in which case every set computes $A$), the collection of sets computing $A$ has measure 0. Recently, more detailed answers have proved to be insightful, in particular in conjunction with answers to the question ``what \emph{kinds} of sets compute $A$?'' For example, as noted above, $A$ is $K$-trivial if and only if $A$ is computed by some set that is $1$-random relative to $A$, in which case the class of oracles computing $A$ is large in an effective sense relative to $A$.

\n {\bf 3.} \emph{Being inert.} Shoenfield's limit lemma states that a set $A$ is computable from the halting set $\emptyset'$ if and only if it has a computable approximation. (We let $\Delta^0_2$ denote the collection of such sets.) The inertness  paradigm  says that  a $\Delta^0_2$ set  $A$ is close to computable if  it is  computably 
approximable with a small number   of    changes.   For formal instances of  the inertness paradigm,  we  use  so-called \emph{cost functions}.  They  measure the total number of  changes of a  $\Delta^0_2$ set, and especially that of a computably enumerable set.  Most examples of cost functions are based on randomness-related concepts.   (Precise definitions of all of these concepts will be given below. For more background on these paradigms see \cite{Nies:ICM}.)

%
%We can vary this concept by either imposing restrictions on the permissible approximations or approximating objects associated with $A$ rather than $A$ itself. As an example of the latter idea, $A$ is low in the traditional sense if and only if $A'$ has a computable approximation. As an example of the former, a set is \emph{$\w$-c.e.}\ if it has a computable approximation with an associated computable bound on the number of changes of this approximation (and $A$ is superlow if and only if $A'$ is $\w$-c.e.). More recently, a systematic study of good approximations has used \emph{cost functions} and the sets that have approximations that \emph{obey} these functions.  

The $K$-trivial sets exemplify these paradigms. Every $K$-trivial set is superlow; as mentioned above, a set is $K$-trivial if and only if it is computable by a set that is $1$-random relative to it;  a set is $K$-trivial if and only if it has an approximation that obeys a canonical cost function $\cost$ defined below.

There are two ways to give mathematical definitions of  lowness properties: combinatorial and analytic. Combinatorial lowness properties, such as (traditional) lowness and superlowness, are defined by discrete tools and by traditional computability. Analytic lowness properties are defined via measure, either directly or coded by prefix-free Kolmogorov complexity, or via some type  of effectively given real number. Even though the central notion of $K$-triviality \emph{implies} some combinatorial lowness properties (such as superlowness), it is only known to be \emph{equivalent} to analytic notions. In other words, currently, $K$-triviality has only analytic characterizations. The search for a combinatorial characterization of the $K$-trivial sets is considered an intriguing open problem. 

\emph{Traceability} is a combinatorial tool that is used to define several lowness properties. Among these notions, \emph{strong jump-traceability}, defined by Figueira, Nies, and Stephan \cite{FigueiraNiesStephan}, was proposed (see \cite{Miller_Nies:_randomness_open_questions}) as a natural candidate for the desired combinatorial characterization of $K$-triviality. This conjecture was refuted by Cholak, Downey, and Greenberg \cite{CholakDowneyGreenberg}; further work \cite{BarmpaliasDowneyGreenberg:orders} refuted another possible characterization, in terms of the rate of growth of the traces. However, Cholak, Downey, and Greenberg did show that for computably enumerable (c.e.)\ sets, strong jump-traceability at least implies $K$-triviality, making strong jump-traceability the first known combinatorial notion to imply $K$-triviality. They further showed that in conjunction with computable enumerability, strong jump-traceability has some appealing structure (it induces an ideal in the Turing degrees). These results prompted interest in strong jump-traceability in its own right. 

 Strong jump-traceability falls under the first lowness paradigm discussed above. since it is  related to weakness of the jump. In \cite{GreenbergNies:benign}, Nies and Greenberg showed that strong jump-traceability can be characterized using cost-function approximations, thus giving it also a characterization within the third paradigm. They used this result to show that every strongly jump-traceable c.e.\ set is computed by many random oracles, a lowness property belonging to the second paradigm. Along the way, they showed that strong jump-traceability is useful in settling problems in other areas of computability, unrelated to randomness. 

\

In the current paper, in a reverse turn of events, we show that strong jump-traceability can in fact be defined analytically, using lowness  properties  of the second paradigm. This result shows the robustness of strong jump-traceability. The heart of the paper, guided by the second paradigm for lowness, is the investigation of the oracular power of random sets relative to c.e.~sets. In other words, the question under consideration is: which c.e.~sets are computable from which random sets?

Two important early results are seminal. Chaitin \cite{Chaitin} showed that there is a \emph{complete} $1$-random set, that is, a $1$-random set that is Turing above the halting set~$\emptyset'$. Hence, every c.e.~set is computable from a $1$-random set. This result was later extended by Ku\v{c}era and G\'{a}cs \cite{Kucera:85,Gacs}, who showed  that \emph{every} set is computable from a $1$-random set. The focus thus turned to \emph{incomplete} $1$-random sets, that is, $1$-random sets that do not compute $\emptyset'$. Here, Ku\v{c}era's basic result \cite{Kucera:86} is that every  $1$-random  $\Delta^0_2$ set is Turing above  a noncomputable c.e.~set. 

The evidence that being computable from an incomplete $1$-random set is a lowness property came much later. In \cite{HirschfeldtNiesStephan:UsingRandomOracles}, Hirschfeldt, Nies, and Stephan showed that if $Y$ is an incomplete $1$-random set, and $A$ is a c.e.~set computable from $Y$, then in fact $Y$ is $1$-random relative to $A$, and $A$ is $K$-trivial. It is still open  whether every $K$-trivial set is computable from some incomplete $1$-random set.

% is considered to be a major open problem in the field of algorithmic randomness; we hope that the results of this paper will also contribute toward a solution to this problem. 

There are two ways to extend Ku\v{c}era's result, both following the second paradigm for lowness. One is to investigate which c.e.~sets are computable by many incomplete sets. Relevant here is the extension of Ku\v{c}era's result by Hirschfeldt and J.~Miller (see \cite[Theorem 5.3.15]{Nbook}), who showed that if $\CC$ is a $\Sigma^0_3$ null class, then there is a noncomputable c.e.~set computed by all $1$-random elements of $\CC$. In the current paper, we show, for several classes $\CC$ of sets, that the strongly jump-traceable c.e.~sets are precisely the c.e.~sets that are computable from all $1$-random elements of~$\CC$. We do this for the classes consisting of the $\w$-c.e.~sets,  the superlow sets, and the superhigh sets. (These and other computability theoretic concepts mentioned in this introduction will be defined below.) These results give   characterizations of c.e.~strong jump-traceability according to the second lowness paradigm. They are  simpler than the characterization according to this paradigm  of  the $K$-trivial sets  as those sets $A$ that are computable from a set that is $1$-random relative to $A$; for strong jump-traceability, we do not need to relativize randomness. 

Another direction for extending Ku\v{c}era's result is to keep our focus on a single random set and the c.e.~sets that it computes. A first attempt would be to  consider a $\Delta^0_2$ $1$-random set $Y$  as ``strong'' if all c.e.~sets computable from $Y$ share some strong lowness property. In early 2009 Greenberg proved \cite{Greenberg} that there is a $\Delta^0_2$ $1$-random set~$Y$ such that every c.e.~set computable from $Y$ is strongly jump-traceable. This result contrasts with the fact, observed in \cite{GreenbergNies:benign}, that no $\w$-c.e.~$1$-random set has this property.

The next logical step to relate the  ``lowness strength'' of a random set with  its degree of randomness. Up to now we have only mentioned the standard notion of randomness, due to Martin-L\"{o}f.  We can investigate what happens if we require a higher level of randomness, that is, if the statistical tests for measuring randomness are made more stringent. The Hirschfeldt-Miller theorem serves as a limiting result, as it implies that a $1$-random set is weakly 2-random (defined in Subsection \ref{subsec:randomndess}) if and only if it computes no noncomputable c.e.~set. Hence we are driven to notions of randomness that are stronger than Martin-L\"{o}f's but still compatible with being $\Delta^0_2$. The natural notion that arises in this context is that of Demuth randomness~\cite{Demuth}. Demuth tests generalize \ML\ tests $(\mathcal G_m) \sN m$   in that one can change the $m$-th component (a $\SI 1$ class  of measure at most $\tp{-m}$) a computably bounded number of times. A set $Z$  fails   a Demuth test if $Z$ is in infinitely many final versions of the $\mathcal G_m$.   In this direction, in mid-2009 Ku\v{c}era and Nies \cite{Kucera.Nies:ta}  extended Greenberg's result by showing that every c.e.~set computable from a Demuth random set is strongly jump-traceable.

 In the current paper we extend the  result of Ku\v{c}era and Nies \cite{Kucera.Nies:ta}  to give a fundamental connection between the second and third lowness paradigms discussed above. We extend the notion of \emph{benign} cost functions, used by Greenberg and Nies to characterize strong jump-traceability, and show the relationship between the strength  of generalized benign cost functions and being computable from       sets of corresponding degree of randomness, as measured by  generalizations of  Demuth randomness.
Translating \Kuc's result to the language of cost functions, Greenberg and Nies have shown that being computable from a $1$-random $\Delta^0_2$ set can be forced by obedience to a corresponding cost function. A variant of a classic result of Ershov implies that the strength of every reasonable cost function can be gauged by some form  of generalized benignity. Putting all of these results together, we get a full correspondence between the second and third paradigms for lowness properties: obedience to cost functions is equivalent to being computable from random $\DII$ sets. For the first time we get an abstract equivalence between paradigms along a wide array of cases, rather than just one witnessed by particular examples. 

\

One question that we have not completely answered is how reliant our results are on the sets investigated being computably enumerable. Several of our implications do not use this hypothesis, but we have not eliminated it completely. There are some preliminary results in this direction. Downey and Greenberg recently managed to eliminate the assumption of computable enumerability from one of the results from \cite{CholakDowneyGreenberg}, showing that every strongly jump-traceable set is $K$-trivial. They conjecture that, like $K$-triviality, the concept of strong jump-traceability is inherently computably enumerable, that is, that the ideal of strongly jump-traceable degrees is generated by its c.e.~elements. If true, this conjecture would imply that almost all of our results carry over to the general, non-c.e.~case.

%%%%%%%%%%%%%%%%%%%%%%%%%%%%%%%%%%%%%%%%%%%%%%%%%%%

\section{Overview}

We give more technical detail on the ideas discussed above, and survey the results of the paper. We also provide    some basic definitions, and fix notation. We assume familiarity with standard computability-theoretic notions and notation.

\subsection{Traceability}

We begin by defining a notion that is central to this paper. An \emph{order function} is a nondecreasing, unbounded computable function $h$ such that $h(0)>0$. A \emph{trace} for a partial function $\psi\colon \w\to \w$ is a uniformly c.e.~sequence $\seq{T_x}$ of finite sets such that $\psi(x)\in T_x$ for all $x\in \dom \psi$. A trace $\seq{T_x}$ is \emph{bounded} by an order function $h$ if $|T_x|\le h(x)$ for all $x$. 

Let $h$ be an order function. A set $A$ is \emph{$h$-jump-traceable} if every partial function that is partial computable in $A$ has a trace that is bounded by $h$. A set $A$ is called \emph{jump-traceable} if it is $h$-jump-traceable for some order function $h$. A set $A$ is called \emph{strongly jump-traceable} if it is $h$-jump-traceable for every order function $h$. 

For every set $A$ there is a universal partial $A$-computable function, which we denote by $J^A$. (We fix $A' = \dom J^A$.) As this universality is witnessed by effective coding, it follows that a set $A$ is jump-traceable if and only if $J^A$ has a trace that is bounded by some order function. Similarly, a set $A$ is strongly jump-traceable if and only if for every order function $h$, the function $J^A$ has a trace that is bounded by $h$.

The class of jump-traceable sets is much larger than the class of strongly jump-traceable sets. There is a perfect class of jump-traceable sets \cite{Nies:Little}, but every strongly jump-traceable set is $\DII$, and indeed $K$-trivial \cite{DowneyGreenberg:sjt2}.

\subsection{Strong reducibilities, $\w$-c.e.~ sets, and superlowness} \label{subsec:superlowness}

A \emph{computable approximation} to a set $A\in 2^\w$ is a uniformly computable sequence $\seq{A_s}_{s<\w}$ such that for every $n$, we have $A_s(n)= A(n)$ for almost all $s$. Associated with every computable approximation $\seq{A_s}$ is the mind-change function $n\mapsto \#\{ s\,:\, A_{s+1}(n)\ne A_s(n)\}$. A set $A$ is \emph{$\w$-c.e.} if it has some computable approximation whose associated mind-change function is bounded by a computable function. 

Let $A$ and $B$ be sets. Recall that $A\lwtt B$ if there is a Turing reduction of $A$ to $B$ with a computable bound on the use of this reduction, and that $A\ltt B$ if and only if $A$ is $B$'s image under a total computable map from $2^\w$ to itself. 

The following are equivalent for a set $A\in 2^\w$: 
\begin{enumerate}
\item $A\lwtt \emptyset'$; 
\item $A\ltt \emptyset'$; 
\item $A$ is $\w$-c.e. 
\end{enumerate}

A set $A$ is \emph{superlow} \cite{BickfordMills, Mohrherr:Refinement} if $A'$ is $\w$-c.e, or equivalently, if every set that is c.e.~relative to $A$ is $\w$-c.e. This formulation points to the fact that this notion does not depend on the choice of enumeration of partial computable functions and hence of universal machine. It also shows immediately that every superlow set is $\w$-c.e.

Nies \cite{Nies:Little} showed that jump-traceability and superlowness coincide on the c.e.~sets, but do not imply each other on the $\w$-c.e.\ sets. (For one direction, by the superlow basis theorem, further discussed in Subsection \ref{subsec:pi01classes}, there is a  superlow $1$-random set. On the other hand, no jump-traceable, or even c.e.~traceable, set can be diagonally non-computable, while each $1$-random set is diagonally non-computable. See \cite{DHbook} or \cite{Nbook} for definitions of these concepts.)

By analogy with the traditional notions of highness and lowness, we define a set $A$ to be \emph{superhigh} \cite{Mohrherr:Refinement}  if $\emptyset''\lwtt A'$,  or equivalently, if $\emptyset''\ltt A'$. This notion too can be characterized in terms of approximations; we discuss this fact in Section \ref{sec:nice}.

\subsection{Measure and randomness} \label{subsec:randomndess}

We let $\lambda$ denote the usual product (``fair coin'') measure on $2^\w$. A (statistical) \emph{test} is a sequence $\seq{\GG_n}_{n<\w}$ of effective (c.e.)\ open subclasses of $2^\w$ such that $\lambda \GG_n \le 2^{-n}$ for all $n$. A set $Z$ \emph{passes} a test $\seq{\GG_n}$ if $Z\notin \GG_n$ for almost all $n$. (The idea is that $\seq{\GG_n}$ determines a null class $\lim \GG_n = \bigcap_{n<\w}\bigcup_{m\ge n} \GG_m$ consisting of the sets that fail the test.)

A test $\seq{\GG_n}$ is a \emph{Martin-L\"{o}f} test if the sequence $\seq{\GG_n}$ is uniformly c.e.; that is, if there is a computable function $f$ such that $f(n)$ is a c.e.~index for $\GG_n$ for all $n$. A set $Z$ is called \emph{Martin-L\"{o}f random}, or \emph{$1$-random}, if it passes every Martin-L\"{o}f test. There is a \emph{universal} Martin-L\"{o}f test; in other words, there is a Martin-L\"{o}f test $\seq{\GG_n}$ such that $\lim \GG_n$ is the collection of sets that are not $1$-random.

These notions can be relativized in the usual computability-theoretic manner, to yield, for instance, the notion of $1$-randomness relative to a given set.

%As mentioned above, we omit the prefix ``Martin-L\"{o}f''; whenever we write ``random'' without further specification, we mean Martin-L\"{o}f random.

%The stronger notion of \emph{weak 2-randomness} is obtained by replacing the condition that $\lambda \GG_n \le 2^{-n}$ for all $n$ by the weaker condition that $\lim_n \lambda \GG_n = 0$.

For more on algorithmic randomness, see \cite{DHbook, DHNT, Nbook}.

\subsection{Characterizations of strong jump-traceability and diamond classes}

The main results of this paper are the characterizations of c.e.~strong jump-trace\-abil\-i\-ty, along the lines of the second paradigm for lowness discussed in the introduction. In Section \ref{sec:main} we give the main argument that establishes the following:
\begin{theorem}
\label{thm:superlow} If a set $A$ is computable from every superlow $1$-random set, then $A$ is strongly jump-traceable. 
\end{theorem}

Here we do not assume that $A$ is c.e. We elaborate on the proof of Theorem \ref{thm:superlow} in Subsection \ref{subsec:pi01classes} below. 

In \cite{GreenbergNies:benign}, it is proved that every c.e.~strongly jump-traceable set is computable from all $\w$-c.e.~$1$-random sets. Since every superlow set is $\w$-c.e., this result, together with Theorem \ref{thm:superlow}, gives us  two characterizations of c.e.~strong jump-traceability. In the following let $A$ be a c.e.\ set. 

\vsps 
\noindent {\bf Characterization Ia.} \emph{$A$ is strongly jump-traceable} $\LR$ 

\emph{ \hfill $A$ is computable from every $\w$-c.e.\ $1$-random set.} \vsps

\vsps

\vsps \n {\bf Characterization Ib.} \emph{$A$ is strongly jump-traceable $\LR$}

\emph{\hfill $A$ is computable from every superlow $1$-random set.} \vsps

For the next characterization of strong jump-traceability, we impose a condition of complexity on the oracle, as opposed to the previous characterizations, where we imposed conditions of simplicity. 

\vsps \n {\bf Characterization II.} \emph{$A$ is strongly jump-traceable $\LR$} 

\emph{\hfill $A$ is computable from every superhigh $1$-random set.} \vsps

We remark that in \cite{Kjos.Nies:nd} it was already shown that some $K$-trivial c.e.\ set is not computable from all superhigh $1$-random sets.  The  result appeared first 
in the conference paper  \cite{Nies:superhighSJT} in extended abstract form. Characterization II is proved in Sections \ref{sec:superhigh1} and \ref{sec:superhigh2}. 

\

A set is LR-hard if $\emptyset'$ is LR-reducible to it. (See \cite{DHbook} or \cite{Nbook} for a definition of LR-reducibility.)
The implication from left to right    of Characterization II improves a result from \cite{GreenbergNies:benign}, that every c.e., strongly jump-traceable set is computable from every LR-hard $1$-random set; Simpson \cite{Simpson} showed that every LR-hard set is superhigh.  It was already noted in \cite{GreenbergNies:benign} that the collection of c.e.~sets that are computable from all LR-hard $1$-random sets strictly contains the strongly jump-traceable ones; it is still open whether this collection coincides with the collection of c.e.~$K$-trivial sets. 

We note that in Characterization Ib, we cannot replace superlowness by lowness, because the only sets that are computable in all $1$-random low sets are the computable sets. This fact can be deduced from a variant of the low basis theorem that allows for upper-cone avoidance (see \cite[Theorem 1.8.39]{Nbook}). Likewise, in Characterization II, we cannot replace superhighness by highness: the $1$-random, high sets $\Om$ and $\Om^{\Halt}$ form a minimal pair, so the only sets that are computable from all $1$-random high sets are the computable sets. (Here $\Omega$ is Chaitin's well-known example of a $1$-random set, and $\Om^{\Halt}$ is its relativization to the halting problem.)

The proofs of the implications of strong jump-traceability (Theorem \ref{thm:superlow} and the right-to-left direction of Characterization II) are of technical interest, as they use a variant of Nies' golden run method that is not, in advance, bounded in depth. This method was developed to show that each $K$-trivial set is low for $K$ (see \cite{Nies:_lowness_and_randomness,Nbook}; for a definition of the concept of lowness for $K$, see \cite{DHbook} or \cite{Nbook}). Some nonuniformity seems to be a key for such an argument (for example, one cannot effectively obtain a constant witnessing lowness for $K$ from a $K$-triviality constant; see \cite{DHbook} or \cite{Nbook} for a discussion of this result). This nonuniformity is amplified in the current constructions.

The following notation will be useful as shorthand. For a class of sets $\CC$, let $\CC^\Diamond$ denote the collection of c.e.\ sets that are computable from all $1$-random sets in~$\CC$. The Hirschfeldt-Miller theorem already mentioned in the introduction  states that if $\CC$ is a null $\Sigma^0_3$ class, then $\CC^\Diamond$ contains a noncomputable (indeed, promptly simple) set. It  extends Ku\v{c}era's classic result that every $\Delta^0_2$ $1$-random set computes a noncomputable c.e.~set, because the singleton $\{Y\}$ is $\Pi^0_2$  for any  $\Delta^0_2$ set  $Y$.  For more background on the diamond operator see \cite[Section 8.5]{Nbook} or \cite{GreenbergNies:benign}. 

The characterizations above can be written as the equalities
\[ (\wce)^\Diamond = \Superlow^\Diamond = \Superhigh ^\Diamond = \SJT ,\]
where $\SJT$ is the collection of c.e., strongly jump-traceable sets. 

We note that every class of the form $\CC^\Diamond$ induces an ideal in the c.e.\ Turing degrees. Hence any of the equalities with $\SJT$ above implies the result from \cite{CholakDowneyGreenberg} that the strongly jump-traceable sets are closed under join.

\subsection{On the assumption of computable enumerability} \label{subsec:ce assumption}

As mentioned in the introduction,  the assumption that the sets in question are computably enumerable is not  used in all  of our characterizations. For example, Theorem \ref{thm:superlow} above  does not rely on such an assumption.  
\begin{enumerate}
	\item We do not know yet how to make good use of strong jump-traceability of a set that is not c.e.; for example, we know only that c.e.\ strongly jump-traceable sets obey all benign cost functions (see below). Thus, in showing that a c.e.\ strongly jump-traceable set is computable from all superlow and superhigh $1$-random sets we make essential use of computable enumerability. 
	
	\item On the other hand, showing that sets that are computable from many $1$-random oracles are strongly jump-traceable, or obey certain cost functions, does not seem to make essential use of the sets being c.e. In all of these examples, the property we use is that the set is superlow and jump-traceable. 
\end{enumerate}

The path from computable enumerability to superlowness and jump-traceability passes through the following fundamental facts, which follow from results in \cite{HirschfeldtNiesStephan:UsingRandomOracles} and \cite{Nies:_lowness_and_randomness}.

\begin{fact} \label{fact:base-for-randomness}
	Suppose that   a set  $A$  is computable from a set that is $1$-random relative to~$A$. (We call such a set a \emph{base for $1$-randomness}.) Then $A$ is superlow and jump-traceable. 
\end{fact}

\begin{proof}
	By \cite{HirschfeldtNiesStephan:UsingRandomOracles}, $A$ is $K$-trivial. By \cite{Nies:_lowness_and_randomness}, $A$ is superlow and jump-traceable. 
\end{proof}

The following appears in \cite{HirschfeldtNiesStephan:UsingRandomOracles}:

\begin{fact} \label{fact:ce-base-for-randomness}
	If $A$ is a c.e.\ set, $Y$ is an incomplete $1$-random set, and $A\leT Y$, then $Y$ is $1$-random relative to $A$, so $A$ is a base for $1$-randomness.
\end{fact}

Combining these results, we have the following:

\begin{corollary} \label{cor:ce-base-for-randomness}
	Every c.e.\ set that is computable from an incomplete $1$-random set is superlow and jump-traceable. 
\end{corollary}

\subsection{Cost functions} \label{subsec:cost-functions}

The third paradigm for lowness states that a $\Delta^0_2$ set $A$ is close to being computable if it has a computable approximation that changes little. Cost functions are the tools that are used to measure this amount of change. For background on cost functions, see \cite[Section 5.3]{Nbook}, \cite{GreenbergNies:benign}, or \cite{Nies:costfunctions}.

A \emph{cost function} is a computable function $c(x,s)$ that takes non-negative rational values. We say that $c$ is \emph{monotone} if $c$ is nonincreasing in the first variable and nondecreasing in the second variable. If $c(x,s)$ is a monotone cost function, then $x\mapsto\lim_s c(x,s) = \sup_s c(x,s)$ is nonincreasing. A cost function $c$ satisfies the \emph{limit condition} if $\lim_s c(x,s)$ is finite for all $x$ and
\[ \lim_{x\to \infty} \sup_s c(x,s) = 0 ,\]
or equivalently, for all $\epsilon>0$, for almost all $x$, we have $c(x,s)<\epsilon$ for all $s$.

Given a computable approximation $\seq{A_s}$ of a set $A$ and a cost function $c$, the \emph{total cost} of the approximation according to $c$ is the quantity 
\begin{equation}
\label{eqn:total sum} 
\sum_{s} c(x,s) \, \Cyl{ x< s \,\, \textrm{is least such that} \, A_{s-1} (x) \neq A_{s} (x)}. 
\end{equation}
We say that a computable approximation $\seq{A_s}$ \emph{obeys} a cost function $c$ if the total cost of $\seq{A_s}$ according to $c$ is finite. The intuitive meaning is that the total amount of changes (as measured by~$c$) is small. We say that a $\Delta^0_2$ set obeys a cost function~$c$ if $A$ has some computable approximation that obeys~$c$. 

The basic result regarding cost functions is that every cost function that satisfies the limit condition is obeyed by some noncomputable (indeed, promptly simple) c.e.~set. This result has its roots in   constructions of \Kuc\ and Terwijn \cite{Kucera.Terwijn:99}, and Downey, Hirschfeldt, Nies, and Stephan \cite{DHNS}. The standard example for a cost function is $\cost(x,s) = \sum_{i=x+1}^s \tp{-K_s(i)}$ (where $K$ is prefix-free Kolmogorov complexity and $K_s$ its stage $s$ approximation).   Nies \cite{Nies:_lowness_and_randomness} characterizes  the $K$-trivial sets   along the lines of the third lowness paradigm: a   set is $K$-trivial if and only if it obeys $\cost$. 

Greenberg and Nies \cite{GreenbergNies:benign} provided a similar result for the c.e.\ strongly jump-traceable sets.  They  introduced a special class of cost functions $c$ with the limit condition:  in an effective  sense   $\sup_s c(v,s)$  converges quickly to  $0$ as $v \ria \infty$. 
    \begin{definition} \label{df:benign}  A    monotonic cost function~$c$ is called    \emph{benign} if there is a computable function $g: \mathbb Q^+ \ria \w$ with the following property: if $0=v_0 < \cdots < v_n$ and $c(v_i, v_{i+1}) \ge q$ for each~$i< n$, then    $n \le g(q)$. \end{definition}

   The main result of   Greenberg and Nies \cite{GreenbergNies:benign} is  that a c.e.~set is strongly jump-traceable if and only if it obeys every benign cost function.    We will apply the harder left-to-right implication several times. We will also improve the right-to-left direction in Corollary~\ref{cor:GNimprove} by discarding the hypothesis that the set is c.e.

\subsection{Equivalence of the second and third lowness paradigms for $\Delta^0_2$ sets} \label{ss:Eq2-3}

As our last main result we show that a particular realization of the second lowness paradigm, being computable from a sufficiently random set, is in a sense equivalent to the third paradigm.
One direction was already obtained by Greenberg and Nies \cite{GreenbergNies:benign}. They defined, for any computable approximation $\seq{Y_s}$ of a $\DII$ set $Y$, a cost function $c_Y$, which satisfies the limit condition, such that if $Y$ is $1$-random, then every c.e.~set that obeys $c_Y$ is computable from $Y$. This construction is essentially a translation of \Kuc's classic argument from \cite{Kucera:86} into the language of cost functions.

Thus, an appropriate cost function forces computability from a given $1$-random $\DII$ set. (If $Y$ is also $\w$-c.e., then $c_Y$ is benign.) Greenberg and Nies then used their characterization of strong jump-traceability in terms of obedience to benign cost functions to obtain their result that $\SJT \subseteq (\wce)^\Diamond$, that is, the left-to-right part of Characterization Ia.

Our last result provides a converse for c.e.~sets. 

\begin{theorem} \label{thm:last} 
	For each cost function $c$ with the limit condition, there is a $1$-random $\DII$ set $Y$ such that each c.e.\ set $A \leT Y$ obeys $c$. 
\end{theorem}

 Demuth randomness is a notion stronger than $1$-randomness that is still compatible with being $\DII$. The $m\tth$  component of a test can be replaced a computably bounded number of times.  In the proof of Theorem~\ref{thm:last}, we gauge how well-behaved  the cost function~$c$ is by associating with it a computable well-ordering $R$. The level of randomness we need to impose on $Y$ in order to obtain the result is given by a further strengthening of Demuth randomness,   where the $m\tth$ component of a test can be changed finitely many times while ``counting down'' along the well-ordering $\w\cdot R$. The details, along with the formal definition of Demuth randomness and its strengthenings, are deferred to  Section \ref{sec:last}.

\

Theorem \ref{thm:last} is related to Characterization Ia, and to Greenberg's result \cite{Greenberg} that there is a $1$-random $\DII$ set $Y$ such that every c.e.~set computable from $Y$ is strongly jump-traceable. As mentioned in the introduction, no $\w$-c.e.~$1$-random set $Y$ can have this property. That is, Characterizations Ia and Ib cannot be replaced by analogous ones involving a single $\w$-c.e., or superlow, $1$-random set, or indeed finitely many such sets. This fact can be argued in two ways. In \cite{GreenbergNies:benign} it is shown that no single benign cost function can force strong jump-traceability, whereas as we already stated, if $Y$ is $\w$-c.e, then $c_Y$ is benign. Alternatively, we can cite work by Ng \cite{Ng:nd}, who showed that the index set of $\SJT$ is $\Pi^0_4$ complete. On the other hand, for any $\Delta^0_2$ set $Y$, the index set $\{e \mid \, W_e \lt Y\}$ is $\Sigma^0_4$. 

However, Characterization Ia does imply the following:

\begin{proposition} \label{prop:single-cost-function} 
There is a monotone cost function $c$ that satisfies the limit condition, such that every  set that obeys $c$ is strongly jump-traceable. 
\end{proposition}

 Proposition \ref{prop:single-cost-function} and Theorem \ref{thm:last} together yield a new proof of Greenberg's result~\cite{Greenberg}.
\begin{proof} The class  of $\w$-c.e.\ sets is $\Sigma^0_3$. Thus, the proof of the 
  the result of  Hirschfeldt and Miller  in  \cite[Theorem 5.3.15]{Nbook} provides a cost function $c$ with the limit condition such that every set $A$ obeying $c$ is Turing below each $1$-random $\w$-c.e.\ set.  By Theorem~\ref{thm:superlow}, every such set $A$  is strongly jump-traceable.   It is easily verified that $c$ is monotone. \end{proof}

%OLD : The cost function $c$ is obtained by combining the cost functions $c_Y$ of \cite{GreenbergNies:benign} for all $\w$-c.e.~sets $Y$. Let $\seq{Y_e}$ be an effective enumeration of all $\w$-c.e.~sets, with fixed approximations $\seq{Y_{e,s}}$, and let $c_{Y_e}$ be the associated cost function derived from the approximation $\seq{Y_{e,s}}$; the functions $c_{Y_e}$ are uniformly computable because the approximations are uniformly computable. Let $c = \sum_e 2^{-e}c_{Y_e}$. Then $c$ satisfies the limit condition because every $c_{Y_e}$ satisfies the limit condition, and every set that obeys $c$ also obeys every $c_{Y_e}$. Hence every c.e.~set that obeys $c$ is computable from every $\w$-c.e.~random set; by Characterization Ia, every such set is strongly jump-traceable. 

  As a corollary to Theorem~\ref{thm:superlow}, we  now improve one direction of the  result of Greenberg and Nies~\cite{GreenbergNies:benign}   stated after Definition~\ref{df:benign}, by  dropping  the hypothesis that the set is c.e.\ 
\begin{corollary}  \label{cor:GNimprove}
Suppose a $\Delta^0_2$ set $A$ obeys all benign cost functions. Then $A$ is strongly jump-traceable. \end{corollary}

\begin{proof}   For each $\omega$-c.e.\ set $Y$, the set  $A$ obeys the benign cost function $c_Y$ defined  in~\cite{GreenbergNies:benign}.   Hence $A \leT Y$. By Theorem~\ref{thm:superlow} this fact implies that $A$ is strongly jump-traceable. \end{proof}

\subsection{Extensions to general $\Pi^0_1$ classes} \label{subsec:pi01classes}

Recall that a \emph{$\Pi^0_1$ class} is the complement of an effectively open (that is, c.e.)\ subclass of $2^\w$; equivalently, it is the collection of paths through a computable subtree of $2^{<\w}$. We make extensive use of the fact that there are $\Pi^0_1$ classes all of whose elements are $1$-random. This fact follows from the existence of a universal Martin-L\"{o}f test.\footnote{A specific example is the class consisting of sets $X$ such that $K(X\uhr n)\ge n-1$ for all $n$; here $K$ denotes prefix-free Kolmogorov complexity.} The collection of $1$-random sets is in fact a union of $\Pi^0_1$ classes. We remark that a $\Pi^0_1$ class consisting only of $1$-random elements cannot  be null. 

The fundamental result regarding $\Pi^0_1$ classes is the Jockusch-Soare superlow basis theorem \cite{Jockusch.Soare:72}, which states that every nonempty $\Pi^0_1$ class contains a superlow element. 

\

The proof of the implication from right to left in Characterization Ib does not make special use of randomness. We will actually prove the following in Section \ref{sec:main}:
\begin{theorem}
\label{thm:Pi01 theorem} Let $\PP$ be a nonempty $\Pi^0_1$ class, and suppose that $A$ is a jump-traceable set computable from every superlow member of $\PP$. Then $A$ is strongly jump-traceable. 
\end{theorem}

Theorem \ref{thm:superlow}, and hence Characterizations Ia and Ib, follow from Theorem \ref{thm:Pi01 theorem} with the aid of the following observation, by applying Theorem \ref{thm:Pi01 theorem} to any $\Pi^0_1$ class that consists only of $1$-random sets.  

\begin{proposition}
\label{prop:} A set that is computable from every superlow $1$-random set is jump-traceable. 
\end{proposition}

\begin{proof}
 By the superlow basis theorem, and   the fact that there is a $\Pi^0_1$ class containing only $1$-random sets,  there is a superlow $1$-random set $Z$. Splitting $Z$ into two halves we can write $Z = X\oplus Y$. By van Lambalgen's Theorem \cite{van-Lambalgen}, both $X$ and $Y$ are $1$-random; indeed, they are relatively $1$-random: $X$ is $1$-random relative to $Y$, and $Y$ is $1$-random relative to $X$. Since $X,Y\leT Z$, both $X$ and $Y$ are superlow. 

By the assumption on $A$, we have $A\leT X,Y$. Since $Y$ is $1$-random relative to $X$, and $A\leT X$, we get that $Y$ is $1$-random relative to $A$. Since also $A\leT Y$, Fact \ref{fact:base-for-randomness} implies  that $A$ is jump-traceable. 
\end{proof}

A further corollary of Theorem \ref{thm:Pi01 theorem} characterizes strong jump-traceability of a c.e.\ set in terms of PA-completeness. (Here a set is PA-complete if it computes a completion of Peano arithmetic, or equivalently, if it computes a member of any nonempty $\Pi^0_1$ class.) Every PA-complete set computes a $1$-random set, and every $\w$-c.e.\ PA-complete set computes an $\w$-c.e.\ $1$-random set.\footnote{There is a Medvedev complete $\Pi^0_1$ class $\PP$ containing a set in every PA-complete degree. There is a $\Pi^0_1$ class $\QQ$ containing only random sets. Since $\QQ$ is Medvedev reducible to $\PP$, for every $X\in \PP$ there is some $Z\in \QQ$ such that $Z\ltt X$. This is because a Turing functional $\Phi\colon \PP\to \QQ$ can be extended to a functional that is total on all sets, and thus a truth-table functional. If $X\in \PP$ is $\w$-c.e., then $X\ltt \emptyset'$, and so for all $Z\ltt X$ we have $Z\ltt\emptyset'$, i.e.\ $Z$ is $\w$-c.e. For background on Medvedev reducibility see \cite{Sorbi:96}}

\begin{corollary}
\label{thm:c.e. PA theorem} A c.e.\ set is strongly jump-traceable if and only if it is computable from every superlow (equivalently, $\w$-c.e.) PA-complete set. 
\end{corollary}

\begin{remark}
For the reverse problem---characterizing the class of sets that are reducible to superlow sets that are PA complete or $1$-random---there is a difference between PA completeness and $1$-randomness. Indeed, every superlow set is computable from \emph{some} superlow, PA complete set: there is a $\Pi^0_1$ class $\PP$ that contains only PA-complete sets (say the class of $\twoset$-valued, diagonally noncomputable functions). By the relativized superlow basis theorem \cite[Exercise 1.8.41]{Nbook}, if $A$ is superlow, then there is some $Z\in \PP$ such that $(A\oplus Z)'\ltt A'$. Now, the class of PA-complete sets is upward closed in the Turing degrees, hence $A\oplus Z$ is PA-complete, is superlow, and computes $A$. So, in short, the class of sets that are computable from PA-complete, superlow sets is exactly the class of superlow sets. 

As mentioned earlier, this assertion is not true if we replace PA-completeness by $1$-randomness; if $A$ is a c.e.~set that is computable in some superlow (indeed, incomplete) $1$-random set, then $A$ is $K$-trivial \cite{HirschfeldtNiesStephan:UsingRandomOracles}, and not every superlow set is $K$-trivial. 
\end{remark}

For many $\Pi^0_1$ classes $\PP$, any set computable in all superlow members of $\PP$ must in fact be computable. For instance, it is not hard to show that there is a $\Pi^0_1$ class $\PP$ without computable members such that any distinct $Y,Z\in \PP$ form a minimal pair. On the other hand, there are $\Pi^0_1$ classes $\PP$ that do not consist only of $1$-random sets or PA-complete sets, such that the class of sets that are computable in all superlow elements of $\PP$ is exactly the class of strongly jump-traceable c.e.\ sets $\SJT$. Consider, for example, the notion of \emph{complex} sets from \cite{Bjorn_etc}. A set $Z$ is complex if there is some order function $h$ such that $C(Z\uhr n)\ge h(n)$ for all $n$ (here $C$ denotes plain Kolmogorov complexity). It was shown in \cite{Bjorn_etc} that a set is complex if and only if there is some diagonally noncomputable function $f$ that is weak-truth-table reducible to $A$. In \cite{Nies:costfunctions}, techniques from \cite{GreenbergNies:benign} are elaborated to show that every c.e., strongly jump-traceable set is computable in any $\w$-c.e.\ complex set. Hence if~$h$ is a sufficiently slow-growing order function, then the class $\PP_h$ of sets $Z$ such that $C(Z\uhr n)\ge h(n)$ for all $n$ is a nonempty $\Pi^0_1$ class with the desired property. 

\

Replacing superlowness by superhighness does not yield a theorem analogous to Theorem \ref{thm:Pi01 theorem}. The reason is that in proving the left-to-right direction of Characterization II, we use the fact that $\Pi^0_1$ classes of $1$-random sets are not null, which allows for \Kuc{} coding into these classes. Not all $\Pi^0_1$ classes admit such coding. However, Medvedev complete classes, such as the class of complete extensions of Peano arithmetic, or of $\twoset$-valued diagonally noncomputable functions, do admit such coding; indeed coding into these classes is easier than into classes of $1$-random sets, because the coding locations can be obtained effectively, essentially by G\"{o}del's incompleteness theorem. Hence, a simpler form  of the argument in Section \ref{sec:superhigh2} would yield the following:

\begin{theorem} \label{thm:superhighPA} 
	Every c.e.~set that is computable from every superhigh PA-complete set is strongly jump-traceable. 
\end{theorem}

Note that this theorem, and the right-to-left direction of Characterization II, can also be viewed as characterizations of the limits of upper-cone avoidance in ``codable'' $\Pi^0_1$ classes in the context of coding. We discuss this idea in Section \ref{sec:superhigh2}.

%%%%%%%%%%%%%%%%%%%%%%%%%%%%%%%%%%%%%%%%%
\section{\texorpdfstring{Restrained approximations}{Restrained approximations of A-partial computable functions}}\label{sec:nice}
Recall that Corollary \ref{cor:ce-base-for-randomness} allows us, in the proof of the right-to-left directions of our characterizations, to replace the assumption that the set $A$ is c.e.\ by the assumption that it is superlow and jump-traceable. In this section, on the way to proving these directions in Sections \ref{sec:main} and \ref{sec:superhigh2}, we comment on this property of the set $A$. We show how this property is exploited to obtain useful approximations for functions that are partial computable in $A$. 

Note that if we assume that $A$ is c.e., then the condition of jump-traceability of Theorem \ref{thm:Pi01 theorem} is guaranteed by the coincidence result in \cite{Nies:Little}, that a c.e.~set is superlow if and only if it is jump-traceable.  

We remark  that this coincidence extends to the $n$-c.e.\ sets by a result of Ng \cite{Ng:thesis}. It is not hard to build a $2$-c.e.\ jump-traceable  (and hence superlow)  set that is not Turing below a c.e.\ jump-traceable set. This fact shows that the class of superlow and jump-traceable degrees is in an essential sense larger than the class of superlow  c.e.\ degrees, and thereby motivates  the extension of some of our results to this case.

\subsection{Bounded limit-recursive functions}

Before we begin, we need to partially relativize the characterization of $\w$-c.e.\ functions mentioned in Subsection \ref{subsec:superlowness}. This relativization will be of use in one direction of Theorem \ref{thm: JT and SL} below, and later when we discuss superhighness in Sections \ref{sec:superhigh1} and \ref{sec:superhigh2}. 

The following definition, due to Cole and Simpson \cite{ColeSimpson:HypMass}, is only a partial relativization of the notion of $\w$-c.e.-ness, because the bound on the number of mind changes remains computable.

\begin{definition}\label{def:BLR}
	Let $X$ be a set. A function $f\colon \w\to \w$ is \emph{bounded limit-recursive} in $X$ (we write $f\in \BLR(X)$) if there is an $X$-computable approximation $\seq{f_s}$ to $f$ such that the associated mind-change function $n\mapsto \#\{s\,:\, f_{s+1}(n)\ne f_s(n)\}$ is bounded by a computable function. 
\end{definition}

Thus a function is $\w$-c.e.\ if and only if it is in $\BLR(\ES)$. The following result, \cite[Theorem 6.4]{ColeSimpson:HypMass}, generalizes the characterization of $\w$-c.e.\ \emph{sets}, but does not generalize to functions. 

\begin{fact} \label{fact:BLR-and-TT}
	Let $X\subseteq \w$. The following are equivalent for a set $A\subseteq \w$:
	\begin{enumerate}
		\item $A\ltt X'$;
		\item $A\lwtt X'$;
		\item $A\in \BLR(X)$. 
	\end{enumerate}
\end{fact}

Hence a set $X$ is superhigh if and only if $\ES''\in \BLR(X)$. 

The next fact, \cite[Corollary 6.15]{ColeSimpson:HypMass}, characterizes the conjunction of superlowness and jump-traceability.

\begin{fact} \label{fact:low-for-BLR}
	A set $X$ is superlow and jump-traceable if and only if $\BLR(X) = \BLR(\ES)$, that is, if and only if every function that is in $\BLR(X)$ is $\w$-c.e.
\end{fact}

\subsection{Functionals}

We define a \emph{partial computable functional} to be a partial computable function $\Gamma\colon 2^{<\w}\times \w\to \w$ such that for all $x<\w$, the domain of $\Gamma(-,x)$ is an antichain of $2^{<\w}$ (in other words, this domain is prefix-free). The idea is that the functional is the collection of minimal oracle computations of an oracle Turing machine. For any $A\in 2^{\le \w}$ and $x<\w$, we let $\Gamma^A(x)=y$ if there is some initial segment $\tau$ of $A$ such that $\Gamma(\tau,x)=y$. Then $\Gamma^A$ is an $A$-partial computable function, and every $A$-partial computable function is of the form $\Gamma^A$ for some partial computable functional $\Gamma$. We write $\Gamma^A(x)\converge$ if $x\in \dom \Gamma^A$; otherwise we write $\Gamma^A(x)\diverge$. The \emph{use} of a computation $\Gamma^A(x)=y$ is the length of the unique initial segment $\tau$ of $A$ such that $\Gamma(\tau,x)=y$.

If $\seq{A_s}$ is a computable approximation to a $\Delta^0_2$ set $A$, and $\seq{\Gamma_s}$ is an effective enumeration of (the graph of) a partial computable functional, then we let $\Gamma^{A}[s] = \Gamma_s^{A_s}$. Note that $\Gamma_s$ is a finite set, and so $\dom \Gamma^A[s]$ is computable, rather than just c.e. By convention, if $\Gamma_s(\tau,x)=y$ then $|\tau|,x,y< s$. 

%Instead of  $\Gamma_s(\tau,x)=y$ we also write $\Gamma_s^\tau(x)=y$ or   $\Gamma^\tau(x)[s]=y$.

\subsection{Existence of restrained approximations}

Let $\seq{A_s}$ be a computable approximation to a $\Delta^0_2$ set $A$, and let $\seq{\Gamma_s}$ be an enumeration of a partial computable functional. We say that $\seq{A_s,\Gamma_s}$ is an approximation to the $A$-partial computable function $\Gamma^A$. 

Suppose that $\Gamma^A(x)\converge[s]$; let $u$ be the use of that computation. We say that this computation is \emph{destroyed} at stage $s+1$ if $A_{s+1}\uhr{u}\ne A_s\uhr{u}$. 

\begin{definition}
\label{def:restrained approx} An approximation $\seq{A_s,\Gamma_s}$ to an $A$-partial computable function is an \emph{restrained $A$-approximation} if there is some computable function $g$ such that for all $x$, the number $g(x)$ bounds the number of stages $s$ such that a computation $\Gamma^A(x)\converge[s]$ is destroyed at stage $s+1$. 
\end{definition}

\begin{theorem} \label{thm: JT and SL} The following are equivalent for a set $A\in 2^\w$: 
\begin{enumerate}
\item $A$ is both superlow and jump-traceable. 
\item Every $A$-partial computable function has a restrained $A$-approximation. 
\end{enumerate}
\end{theorem}

\begin{proof}
The easier implication is $\mbox{(2)}\Then\mbox{(1)}$. Let $C$ be a set that is c.e.~in~$A$. There is an $A$-partial computable function $\theta$ such that $C= \dom\theta$. Let $\seq{A_s,\Gamma_s}$ be a restrained $A$-approximation to $\theta$, witnessed by a computable function $g$. Let $C_s(x)=1$ if $\Gamma^{A}(x)\converge[s]$; otherwise let $C_s(x)=0$. Then $\seq{C_s}$ and $g$ show that $C$ is $\w$-c.e. Hence $A$ is superlow.

Let $\theta$ be $A$-partial computable. Let $\seq{A_s,\Gamma_s}$ be a restrained $A$-approximation to $\theta$, witnessed by a computable function $g$. Let
\[ T_x = \left\{ \Gamma^{A}(x)[s]\,:\, s<\w\,\,\,\&\,\,\,\Gamma^{A}(x)\converge[s] \right\} .\]
Then $|T_x| \le g(x)+1$ for every $x$, and $\theta(x)=\Gamma^A(x)\in T_x$ for all $x\in \dom \theta$.

Hence every $A$-partial computable function has a trace bounded by some computable function. As discussed in subsection \ref{subsec:superlowness}, $A$ is jump-traceable. 

\

We now turn to the proof of the converse implication $\mbox{(1)}\Then\mbox{(2)}$. 
%In \cite{ColeSimpson:HypMass}, Cole and Simpson define, for every set $X\in 2^\w$, the class of functions $\textup{BLR}(X)$ (the $X$-\emph{bounded limit recursive} functions): these are the functions that have an $X$-computable approximation $\seq{f_s}$ whose associated mind-change function is bounded by a computable function. Hence $\textup{BLR}(\emptyset)$ is the class of $\w$-c.e.\ functions.
% Cole and Simpson \cite[Cor.\ 6.15]{ColeSimpson:HypMass} showed that each function in $\textup{BLR}(X)$ is $\w$-c.e.\ if and only if $X$ is both jump-traceable and superlow.
Let $A$ be a superlow, jump-traceable set, and let $\theta$ be an $A$-partial computable function. Let $\widetilde \Gamma$ be a Turing functional such that $\widetilde \Gamma^A = \theta$. 

For $x\in \dom \theta$, let $f(x) = A\uhr u$, where $u>0$ is the use of the computation $\widetilde \Gamma^A(x)$; for all $x\notin \dom \theta$, let $f(x)$ be the empty string. Then $f\in \textup{BLR}(A)$: indeed, $\dom \theta$ is $A$-c.e., and so we can approximate $f$ computably in $A$, changing our mind only once for $x\in \dom \theta$, and not at all for $x\notin\dom \theta$.

By Fact \ref{fact:low-for-BLR}, $f$ is $\w$-c.e. Let $\seq{f_s}$ be a computable approximation to $f$, with a mind-change function that is bounded by some computable function $g$. Let $x_s \le s$ be the largest $x$ such that
\[\forall y , z \le x \, [ f_s(y) \subseteq f_s(z) \ \text{or} \ f_s(z) \subseteq f_s(y) ],\]
and let $\sigma_s= \bigcup_{y \le x_s} f_s(y)$. Then $\lim_s x_s = \infty$ and for each $n$, for almost all $s$ we have $A\uhr n \subseteq \sigma_s$. Let $A_s(y) = \sigma_s(y)$ for $y < |\sigma_s|$ and $A_s(y) =0$ otherwise. Then $\seq{A_s}$ is a computable approximation to $A$. 

Let $\seq{\widetilde \Gamma_s}$ be some computable enumeration of the Turing functional $\widetilde \Gamma$. Now define an enumeration $\seq{\Gamma_s}$ of a partial computable functional $\Gamma\subseteq \widetilde \Gamma$ as follows: at stage $s$, if the axiom $(\s,x)\mapsto y$ is already in $\widetilde \Gamma_s$, enumerate that axiom into $\Gamma_s$ if $x\le x_s$ and $\s\subseteq f_s(x)$. 
% \begin{center} $\Gamma_s^{\sigma}(x) = b$ if  $\widetilde \Gamma_s^{\sigma}(x) = b$, $x \le x_s$ and $\sigma \subseteq  f_s(x)$. \end{center} 
Then $\Gamma^A = \theta$. To show that $\seq{A_s,\Gamma_s}$ is a restrained approximation, note that if $\Gamma^A(x)\converge[s]$ with use $u$, and that computation is destroyed at stage $s+1$, then $u \le |f_s(x)|$, and $f_s(y) \neq f_{s+1}(y)$ for some $y \le x$. So the number of times this event can happen is bounded by $\sum_{y \le x} g(y)$. 
\end{proof}

%%%%%%%%%%%%%%%%%%%%%%%%%%%%%%%%%%%%%%%%%
\section{$\SJT$ coincides with $\Superlow^\Diamond$} \label{sec:main}

In this section we prove Theorem \ref{thm:Pi01 theorem}. As explained above, together with the results in \cite{GreenbergNies:benign}, this theorem provides Characterizations Ia and Ib of the strongly jump-traceable c.e.\ sets. 
We first fix some notation. 

\subsection{Notation for classes of sets} \label{subsec:notation}

For a finite binary string $\s\in 2^{<\w}$, we let $[\s]$ denote the sub-basic clopen subclass of $2^\w$ consisting of all infinite binary strings that extend $\s$. If $W$ is a c.e.~subset of $2^{<\w}$, then we let $\Opcl W  = \bigcup_{\s\in W}[\s]$ be the effectively open subset of $2^\w$ determined by $W$. (A clopen class $\Opcl C $ is determined by a finite set of strings $C$.)

A $\Pi^0_1$ class is the complement of some effectively open subclass of $2^\w$. A \emph{$\Pi^0_1$ index} for a $\Pi^0_1$ class $\PP$ is a c.e.~index for a c.e.~set $W\subseteq 2^{<\w}$ such that $\PP = 2^\w \setminus \Opcl W $. 

%From a canonical effective enumeration of all c.e.~sets, we thus obtain an effective enumeration of all $\Pi^0_1$ classes. 
A $\Pi^0_1$ class $\PP$ admits an approximation $\PP = \bigcap_t \PP_t$, where $\seq{\PP_t}$ is a computable sequence of clopen subsets of $2^\w$. Namely, we let $\PP_t = 2^\w \setminus \Opcl{W_s}$, where $\seq{W_s}$ is an effective enumeration of the c.e.~set $W$ such that $\PP = 2^\w \setminus \Opcl W $. (Alternatively, we can fix a computable tree $T\subseteq 2^{<\w}$ such that $\PP$ is the collection of paths through $T$, and let $\PP_t$ be the union of $[\s]$ where $\s\in T$ has length $t$.) Given a $\Pi^0_1$ index for $\PP$, we can effectively obtain the approximation $\seq{\PP_t}$. 

We will make fundamental use of the compactness of $2^\w$, which implies that if $\PP$ is an empty $\Pi^0_1$ class, then there is some $t$ such that $\PP_t$ is empty.

\subsection{Discussion of the proof of Theorem \ref{thm:Pi01 theorem}}

Let $A$ be a jump-traceable set that is computable from every superlow member of a nonempty $\Pi^0_1$ class $\PP$. As mentioned earlier, by the superlow basis theorem, $A$ is superlow. By Theorem \ref{thm: JT and SL}, every $A$-partial computable function has a restrained $A$-approximation.

We will show that for every order function $h$, every $A$-partial computable function has a trace bounded by $x\mapsto \tp{h(x)}$. This fact suffices for the strong jump-traceability of $A$ since $h$ can be an arbitrary order function. For the rest of this section, fix an order function $h$, and fix an $A$-partial computable function $\theta$. Let $\seq{A_s,\Gamma_s}$ be a restrained $A$-approximation for $\theta$, witnessed by a computable function $g$ (as in Definition~\ref{def:restrained approx}).

\

The strategy for obtaining a trace for $\theta$ is to try, and fail, to construct a superlow set $Z\in \PP$ such that $A\nle_\Tur Z$. Let $\seq{\Phi_e}$ be an effective enumeration of all Turing functionals. For each $e$, we attempt to meet the requirement $A \ne \Phi_e(Z)$. Overall, the construction consists of a recursive calling of strategies (or procedures); the strategy $R^e$ which attempts to meet the $e\tth$ requirement $A\ne \Phi_e(Z)$ is located at the \emph{$e\tth$ level} of the structure of all called strategies. 

We recall the proof of the Jockusch-Soare superlow basis theorem \cite{Jockusch.Soare:72}. A superlow element of a given nonempty $\Pi^0_1$ class $\QQ$ is obtained by recursively defining a sequence of decreasing subclasses of $\QQ$, each deciding the next element of the jump. Given $\QQ$, we let $\QQ\seq{0} = \QQ$, and
\[ \QQ\seq{n+1} = 
\begin{cases}
\QQ\seq{n}, & \text{ if $n\in X'$ for all $X\in \QQ\seq{n}$} \\
\left\{ X\in \QQ\seq{n}\,:\, n\notin X' \right\} & \text{ otherwise.} 
\end{cases}
\]
Then $\bigcap_n \QQ\seq{n}$ is a singleton $\{Z\}$ where $Z$ is superlow.\footnote{To see that, we approximate the sequence $\seq{\QQ\seq{n}}$. For a finite binary string $\alpha\in 2^{<\w}$, recursively define a subclass $\QQ\seq{\alpha}$ of $\QQ$ as follows: let $\QQ\seq{} = \QQ$; given $\QQ\seq{\alpha}$, let $\QQ\seq{\alpha1} = \QQ\seq{\alpha}$, and let $\QQ\seq{\alpha0} = \left\{ X\in \QQ\seq{\alpha}\,:\, |\alpha|\notin X' \right\}$. At stage $s$ of an effective construction, we define $\alpha_s\in 2^{<\w}$ to be the leftmost binary string $\alpha$ of length $s$ such that $(\QQ\seq{\alpha})_s$ is nonempty. If $s<t$, then $(\QQ\seq{\alpha})_s\supseteq (\QQ\seq{\alpha})_t$ for all $\alpha$, so $\alpha_t$ does not lie to the (lexicographic) left of $\alpha_s$. Hence the total number of stages $s$ such that $\alpha_s\uhr n \ne \alpha_{s+1}\uhr n$ is at most $2^n$.}

Cone avoidance, that is, meeting the requirements $A\ne \Phi_e(Z)$, can also be obtained in a similar fashion (``forcing with $\Pi^0_1$ classes''): we intersect the given class with one of the classes $\RR_{e,\tau} = \{X\,:\, \Phi_e(X)\nsupseteq \tau\}$ for some finite initial segment $\tau$ of $A$. Thus we attempt to intersperse these classes with classes as above for the superlowness of $Z$. The assumption on $A$ implies that this attempt will fail. The failure is due to the fact that at some level $e$, all attempts to diagonalize $\Phi_e(A)$ away from an initial segment of $A$ yield empty $\Pi^0_1$ classes. This fact gives us a method for confirming ``believable'' computations $\Gamma^A(x)\converge[s]$, and hence building a trace for $\theta$. 

The combinatorial content of the construction is showing how to effectively approximate this final outcome, as in the computable approximation of the forcing proof of the superlow basis theorem. We need to show that if the attempts to build a trace fail, that is, if all the requirements are met, then the set $Z$ constructed is indeed superlow. 

Fix a level $e$. For each $x$, a strategy $S^e_x$ is responsible for confirming computations $\Gamma^A(x)\converge[s]$. Say such a computation appears, with some use $u$. The strategy $S^e_x$ tests whether $A_s\uhr u$ is really an initial segment of $A$ by attempting to meet the $e$-th requirement by intersecting the current class with the class $\RR_{e,A_s\uhr u}$. The strategy then waits for the resulting intersection to become empty; if $A$ moves in the meantime, the computation $\Gamma^A(x)[s]$ is destroyed and no harm is done. As long as the class is not empty, it seems like the $e$-th requirement is met, and so a new strategy for meeting the $(e+1)$-st requirement is called in the meantime, starting a new superlow basis construction within that $\PPI$ class. If the resulting class turns out to be empty, $A_s\uhr u$ is confirmed and the computation $\Gamma^A(x)[s]$ traced. 

To show that the construction succeeds, we then argue for a contradiction and assume that at all levels $e$, some strategy $S^e_x$ succeeds in meeting the $e$-th requirement. The key, as mentioned, is to ensure that the resulting set $Z$ is superlow, even though the superlowness strategies are distributed over all the levels of the construction. Premature changes in $A$ may cause difficulties here. Say a strategy $S^e_x$ calls a procedure $R^{e+1}$ while trying to certify a computation $\Gamma^A(x)[s]$; this run of $R^{e+1}$ may then be cancelled due to an $A$ change that destroys that computation. This cancellation may in turn change our approximation to $Z'$. To put a computable bound on the number of times such an event can occur, we use the fact that $\seq{A_s,\Gamma_s}$ is restrained.

\subsection{Golden pairs}

Say that the construction above succeeds at a level $e$. The following definition captures the relevant properties of the final $\Pi^0_1$ class $\QQ$ that is passed to the successful run of $R^e$, and of the associated Turing functional $\Phi_e$. We again use the notation $\seq{\QQ\seq{n}}$ to denote the sequence of $\Pi^0_1$ classes obtained in the proof of the superlow basis theorem. 

\nicebox{
\begin{definition}
\label{def:golden pair} 
A pair $\QQ,\Phi$, consisting of a nonempty $\Pi^0_1$ class and a Turing functional, is a \emph{golden pair} for $\Gamma$ and $h$ if for almost all $x$ such that $\Gamma^A(x)\converge$, with some use $u$, for all $X\in \QQ\seq{h(x)}$ we have $\Phi(X)\supseteq A\uhr{u}$. 
\end{definition}  }
\vsps 

The proof that $\theta$ has a trace bounded by $h$ is split into two separate propositions. The first verifies that golden pairs indeed yield traces. 

\begin{proposition}
\label{prop:golden pairs do the trick} If there is a golden pair for $\Gamma$ and $h$, then $\theta = \Gamma^A$ has a trace $\seq {V_x }$ such that $|V_x| \le \tp{h(x)}$ for each $x$. 
\end{proposition}

The second proposition asserts the existence of a golden pair. 

\begin{proposition}
\label{prop:golden pairs exist} If $A$ is computable from every superlow member of $\PP$, then there are a $\Pi^0_1$ class $\QQ\subseteq \PP$ and a functional $\Phi$ such that $\QQ,\Phi$ is a golden pair for $\Gamma$ and $h$. 
\end{proposition}

\begin{proof}
[Proof of Proposition \ref{prop:golden pairs do the trick}] Let $\QQ,\Phi$ be a golden pair for $\Gamma$ and $h$. We let the $\PPI$ class $\QQ\seq{n}[s]$ be the stage $s$ approximation to $\QQ\seq{n}$. It is defined inductively like $\QQ\seq{n}$, but assessed with the information present at stage $s$. That is, if $\Phi_e^\s(e)\converge$ for each $\s$ such that $[\s] \subseteq (\QQ\seq{n}[s])_s$ (that is, $n\in X'$ for all $X\supset \s$), then we let $\QQ\seq{n+1}[s]=\QQ\seq{n}[s]$; otherwise, we let $\QQ\seq{n+1}[s] = \left\{ X\in \QQ\seq{n}[s]\,:\, n\notin X' \right\}$. As mentioned above, for every $n$, there are at most $2^n$ many $\Pi^0_1$ classes that are ever chosen to be $\QQ\seq{n}[s]$.

We enumerate a number $y$ into a set $V_x$ at stage $s$ if at that stage we discover that there is a binary sequence $\tau$ such that $\Gamma_s(\tau,x)=y$ and such that $\Phi(X)$ extends $\tau$ for every $X\in (\QQ\seq{h(x)}[s])_s$ (which means that for all strings $\s$ of length $s$ such that $[\s] \sub (\QQ\seq{h(x)}[s])_s$, we have $\Phi(\s)\supseteq \tau$).

It suffices to show that some finite variant of $\seq{V_x}$ is a trace for $\theta$ that is bounded by $2^h$. The sequence $\seq{V_x}$ is uniformly c.e. For any version of $\QQ\seq{h(x)}[s]$, at most one number $y$ gets enumerated into $V_x$, so $|V_x|$ is bounded by the number $2^{h(x)}$ of possible choices for $\QQ\seq{h(x)}[s]$. Finally, for almost all $x\in\dom \theta$, for large enough $s$, for every $X\in \QQ\seq{h(x)}[s]=\QQ\seq{h(x)}$ we have $\Phi(X)\supseteq \tau = A\uhr{u}$, where $u$ is the use of the computation $\Gamma^A(x)$. Then $\theta(x)\in V_x$ for almost all $x \in \dom\theta$. 
\end{proof}

\subsection{A golden pair exists}

The heart of the proof of Theorem \ref{thm:Pi01 theorem} is the proof of Proposition \ref{prop:golden pairs exist}: that under the assumptions on $A$ and $\PP$, a golden pair exists for $A$ and $h$. As already mentioned in the introduction, the mechanism is a nonuniform argument in the spirit of the golden run method from \cite{Nies:_lowness_and_randomness}, except that the procedure-calling structure now has unbounded depth.

The argument was sketched already in our discussion leading to the definition~\ref{def:golden pair} of golden pairs. For every $e$, a procedure $R^e$, provided with some $\Pi^0_1$ subclass $\PP^e$ of $\PP$ as input, attempts to show that $\PP^e,\Phi_e$ is a golden pair for $A$ and $h$. For each $x<\w$, if $\Gamma^A(x)\converge$ with use $u$, then a subprocedure $S^e_x$ wants to either give permanent control to the next level~$e+1$, or show that the golden pair condition holds at $x$ for $\QQ=\PP^e$: for all $X\in \PP^e\seq{h(x)}$ we have $\Phi_e(X)\supseteq A\uhr{u}$.

\subsubsection*{The procedures and the construction} 

A typical procedure calling structure at any stage of the construction is 
\begin{center}
$R^0 \rra S^0_y \rra \cdots \rra R^e \rra S^e_x \rra R^{e+1} \rra \cdots $ 
\end{center}

The instructions for our procedures are simple. 

\n \emph{Procedure $R^e$.} This procedure runs with input $\PP^e$ (a $\Pi^0_1$ class) and a parameter $n<\w$. While $R^e$ is running, every number $x$ such that $h(x)>n$ is marked as either \emph{fresh} or \emph{confirmed}. At the inception of $R^e$, all numbers $x$ such that $h(x)>n$ are marked as fresh. 

If $R^e$ has control at some stage $s$, and there is some $x$ that is fresh at stage $s$ and such that $\Gamma^A(x)\converge\![s]$ with use $u<s$, then for the least such $x$, we call a subprocedure $S^e_x$ with input $\tau = A_s\uhr{u}$. 

\n \emph{Procedure $S^e_x$.} A run of this procedure is provided with a string $\tau$---an initial segment of the current state of $A$---that witnesses that $\Gamma^A(x)\converge\![s]$. It acts as follows. 
\begin{enumerate}
\item[(a)] Start a run of $R^{e+1}$, with the input 
\begin{equation*}
%\label{eqn:Pnew}  
\PP^{e+1} = \left\{ X\in \PP^e\seq{h(x)}\,:\, \Phi_e(X)\nsupseteq \tau \right\},
\end{equation*}
and parameter $h(x)$. 

As long as we do not see that for every $X\in \PP^e\seq{h(x)}$ we have $\Phi_e(X)\supseteq \tau$, that is, as long as $\PP^{e+1}$ appears to be nonempty (and so $S^e_x$ has not yet succeeded), we halt all activity for $R^e$ and let the run of $R^{e+1}$ take its course.

\item[(b)] If we see that $\PP^{e+1}$ is empty, we cancel the run of $R^{e+1}$ (and any of its subprocedures), and return control to $R^e$, marking $x$ as confirmed. 
\end{enumerate}

A run of $S^e_x$ started at a stage $s$ with input $\tau = A_s\uhr u$ believes that $\PP^e\seq{h(x)}[s] = \PP^e\seq{h(x)}$ (and indeed that $\PP^e\seq{h(x)}[s]= \PP^e\seq{h(x)}[t]$ for all $t>s$), and that $\tau \subset A$. If either of these beliefs is incorrect, then we let $t$ be the least stage at which we discover this incorrectness: either $\tau\not\subset A_t$, or $\PP^e\seq{h(x)}[t]\ne \PP^e\seq{h(x)}[t-1]$. If $S^e_x$ is still running at stage $t$, then we immediately cancel it (along with the run of $R^{e+1}$ it called and all of its subprocedures), and return control to $R^e$. If $S^e_x$ has already returned control to $R^e$, then we re-mark $x$ as fresh at stage $t$. 

\

The entire construction is started by calling $R^0$ with input $\PP^0 = \PP$ and parameter~0.

\

\subsubsection*{Verification} 

We show that there is some $e$ such that $\PP^{e},\Phi_e$ is a golden pair for $\Gamma$, $h$ (for some stable version of $\PP^{e}$). A \emph{golden run} is a run of a procedure $R^{e}$ that is never cancelled, such that every subprocedure $S^{e}_x$ that is called by that run eventually returns or is cancelled. 

\begin{claim} \label{claim;GRGP 1} 
	If there is a golden run of $R^e$ with input $\QQ$, then $\QQ,\Phi_e$ is a golden pair for $\Gamma$ and $h$. 
\end{claim}

\begin{proof}
Suppose the golden run of $R^e$ is called with parameter $n$. Note that its input $\QQ$ is the final version of $\PP^e$. Since $h$ is an order function, for almost all $x$ we have $h(x)>n$. 

We show that for every $x$, only finitely many runs of $S^e_x$ are ever called. Let $x<\w$ be such that $h(x)>n$. If $x\notin \dom \theta$, then since $\seq{A_s,\Gamma_s}$ is a restrained approximation, we have $\Gamma^A(x)\converge\![s]$ for only finitely many stages $s$. Thus, in this case, there is a stage after which no run of $S^e_x$ is called. 

Suppose that $x\in \dom \theta$. Let $u$ be the use of the computation $\Gamma^A(x)$. For sufficiently late $s$ we have $A_s\uhr u \subset A$ and $\QQ\seq{h(x)}[t] = \QQ\seq{h(x)}[s]$ for all $t>s$. If a run of $S^e_x$ is called at such a late stage $s$, then it will never be cancelled. When it returns, $x$ will be marked confirmed, and never re-marked fresh; hence no later run of $S^e_x$ will ever be called. 

A similar argument shows that if $x\in \dom \theta$ and $h(x)>n$, then a run of $S^e_x$ will indeed be called and never cancelled: We can wait for a stage $s$ that is late enough so that the conditions above hold and, in addition, $\Gamma_s(A_s\uhr{u},x)= \theta(x)$ and no run $S^e_y$ for any $y<x$ is ever called after stage $s$. If $x$ is marked fresh at such a stage $s$, then a run of $S^e_x$ will be called and never cancelled. Since the run of $R^e$ is golden, such a run will return, and $x$ will be marked confirmed and never re-marked fresh. 

Let $x\in \dom \theta$ be such that $h(x)>n$. Let $s$ be the stage at which the last run of $S^e_x$ is called. As we just argued, this run is not cancelled; it returns at some stage $t>s$, and $x$ is confirmed at all stages after $t$. We thus have $\QQ\seq{h(x)} = \QQ\seq{h(x)}[s]$, and $A_s\uhr {u} \subset A$, where $u$ is the use of the computation $\Gamma^A(x)$. At stage $t$ we witness the fact that $\Phi_e(X)\supseteq A_s\uhr u$ for all $X\in \QQ\seq{h(x)}$. 

Thus $\QQ, \Phi_e$ is a golden pair as required. 
\end{proof}

It remains to show that there is a golden run of some $R^e$. We first need to do some counting, to establish a computable bound $N(x)$ on the number of times a procedure $S^e_x$ (for any $e$) is called. We then argue as follows. Suppose there is no golden run, so every run of every $R^e$ is either eventually cancelled, or calls some run of $S^e_x$ that is never cancelled but never returns. By induction on $e$ we can see that for every $e$, there is a run of $R^e$ that is never cancelled, with a final version of $\PP^e$. The sequence of $\Pi^0_1$ classes $\PP^0,\PP^1,\dots$ is nested, and so its intersection $\bigcap_e \PP^e$ is nonempty. Let $Z\in \bigcap_e \PP^e$. We will show that we can use approximations to the trees $\PP^e$ to computably approximate $Z'$, and that we can use our computable bounds on the number of times procedures can be called to ensure a computable bound on the number of changes. Hence $Z$ is superlow. By our hypothesis on $A$, there will be some $e$ such that $\Phi_e(Z)=A$. Consider the run of $S^e_x$ that is never cancelled nor returns, which defines the last version of $\PP^{e+1}$. It defines
\[ \PP^{e+1} = \left\{ X\in \PP^e\seq{h(x)}\,:\, \Phi_e(X)\nsupseteq \tau \right\},\]
where $\tau\subset A$ (since $S^e_x$ is never cancelled). But this definition contradicts the fact that $Z\in \PP^{e+1}$.

We now give the details of this argument. Recall that $g(x)$ is the computable function from Definition~\ref{def:restrained approx} bounding how often a computation $\Gamma^A(x)$ can be destroyed.

\begin{claim} \label{claim: local counting} 
	For each $e$ and $x$, every run of $R^e$ calls at most $g(x)+2^{h(x)}$ many runs of~$S^e_x$. 
\end{claim}

\begin{proof}
Suppose that at stage $s$, a run of $S^e_x$ is cancelled while the run of $R^e$ that called it is not cancelled. Let $\PP^e$ be the input of this run of $R^e$, and let $\tau$ be the input of $S^e_x$.

One of the following possibilities holds: 
\begin{enumerate}

\item[(a)] $\PP^e\seq{h(x)}[s]\ne \PP^e\seq{h(x)}[s-1]$; or

\item[(b)] $\tau\subset A_{s-1}$ but $\tau\not\subset A_s$. 
\end{enumerate}

The first possibility occurs fewer than $\tp{h(x)}$ many times. The second, by the fact that $\seq{A_s,\Gamma_s}$ is a restrained approximation for $\theta$, occurs at most $g(x)$ many times. 
\end{proof}

\begin{claim} \label{claim: global counting}
	There is a computable bound $N(x)$ on the number of times a procedure $S^e_x$ is called for any $e$. 
\end{claim}

\begin{proof}
We calculate, by recursion on $e$ and $x$, a bound $M(e,x)$ on the number of times any run of $R^e$ calls a run of $S^e_x$. We use Claim \ref{claim: local counting}. Since there is only one run of $R^0$, we can let $M(0,x)= g(x)+2^{h(x)}$. For $e>0$ we let $M(e,x)$ be the product of $g(x)+2^{h(x)}$ with a bound on the number of runs of $R^e$ that are called by some $S^{e-1}_y$ with parameter $h(y)<h(x)$. 

Since $h(y)<h(x)$ implies $y<x$, the number of runs of $R^e$ with a parameter less than $h(x)$ is bounded by
\[ \sum_{y<x} M(e-1,y) .\]
This completes the recursive definition of $M$. Now, by induction on $e$, the parameter of any run of $R^e$ is at least~$e$. So we can let $N(x)= \sum_{e<h(x)} M(e,x)$. 
\end{proof}

Now suppose for a contradiction that there is no golden run. So every run of every $R^e$ is either eventually cancelled, or calls some run of $S^e_x$ that is never cancelled but never returns. As mentioned above, by induction on $e$ we can see that for every $e$, there is a run of $R^e$ that is never cancelled, with a final version of $\PP^e$. 

The sequence of $\Pi^0_1$ classes $\PP^0,\PP^1,\dots$ is nested, and so its intersection $\bigcap_e \PP^e$ is nonempty. Let $Z\in \bigcap_e \PP^e$.

\begin{claim} \label{claim: Z is superlow} 
	$Z$ is superlow. 
\end{claim}

\begin{proof}
Let $n>0$, and let $e$ be the least number such that the permanent run of $R^e$ is started with a parameter greater than $n$. As mentioned during the proof of Claim \ref{claim: global counting}, the parameter of any run of $R^e$ is at least $e$, so such an $e$ exists.

Whether $n\in Z'$ depends only on $\PP^{e-1}\seq{n+1}$. So we can approximate an answer to the question of whether $n\in Z'$ by tracking, at a stage $s$, the definition of $\PP^{d}\seq{n+1}$ at that stage, where $d$ is the greatest number such that the current (at stage $s$) run of $R^d$ was started with a parameter $h(x) \le n$.

The current version of $\PP^d\seq{n+1}$ can change because we call the procedure $S^e_x$ for some $h(x)~\le~n$. Otherwise it can change due to the approximation feature of the proof of the superlow basis theorem (see the proof of Proposition \ref{prop:golden pairs do the trick}). Thus the number of changes is bounded by
\[ 2^{n+1} \sum_{h(x) \le n} N(x),\]
which is a computable bound. Thus the above procedure gives an $\w$-c.e.\ approximation to $Z'$. 
\end{proof}

By the assumption on $A$, we have $A\le_\Tur Z$. Hence there is some $e$ such that $\Phi_e(Z)=A$. Consider the run of $S^e_x$ that is never cancelled nor returns, which defines the last version of $\PP^{e+1}$. It defines
\[ \PP^{e+1} = \left\{ X\in \PP^e\seq{h(x)}\,:\, \Phi_e(X)\nsupseteq \tau \right\},\]
where $\tau\subset A$. As already explained above, this definition contradicts the fact that $Z\in \PP^{e+1}$. This completes the proof of Proposition \ref{prop:golden pairs exist} and so of Theorem \ref{thm:Pi01 theorem}.

\begin{remark}
To show that $A$ is strongly jump-traceable, it is sufficient to show that for every order function $h$, a universal $A$-partial computable function $\theta$ has a trace bounded by $h$. The reader may wonder why we bother with every $A$-partial computable function, rather than just a universal one. Let $J$ be a partial computable functional such that for all sets $X$, the function $J^X$ is a universal $X$-partial computable function. Even though $\theta=J^A$ is universal, the restrained $A$-approximation for $\theta$ gives a partial computable functional $\Gamma$ such that $\Gamma^A=\theta$, but for other sets $X$ it will not be the case that $\Gamma^X$ is universal for $X$-partial computable functions. In the proof, it is the approximation $\Gamma^A\,[s]$ that we use, not $J^A\,[s]$, so we might as well work with a general function, rather than just a universal one. 
\end{remark}

%%%%%%%%%%%%%%%%%%%%%%%%%%%%%%%%%%%%%%%%%%%%
\section{$\SJT$ is contained in $\Superhigh^\Diamond$} \label{sec:superhigh1}

In this section we provide the left-to-right direction of Characterization II of the strongly jump-traceable c.e.\ sets: every c.e., strongly jump-traceable set is computable from every superhigh $1$-random set. 

In fact, we prove a slightly stronger result, Theorem \ref{prop:left2right}, by replacing the class of superhigh sets by a a larger null $\SIII$ class $\HHH$, introduced by Simpson \cite{Simpson}, which is related to PA-completeness. We actually show that every strongly jump-traceable c.e.\ set is in $\HHH^\Diamond$.

 To define $\HHH$, recall that a function~$f$ is diagonally non-computable (d.n.c.)\ relative to $Y$ if for all $x\in \dom J^{Y}$, we have $f(x)\ne J^{Y}(x)$. (Recall  also that   $J$   denotes  a partial computable functional such that for every  set $Y$, the function $J^Y$ is a universal $Y$-partial computable function.)
 
Let~$\PP$ be the $\PPI(\ES')$ class of $\twoset$-valued functions that are d.n.c.\ relative to $\ES'$. By a result of Jockusch \cite{JockuschDNR} %\cite[Ex.\ 5.1.15]{Nbook}) 
relativized to $\ES'$, the class
\[ \left\{Z \,\colon \, \ex f \leT Z \oplus \ES' \,\, [f\in \PP]\right\}\]
is null. The class $\GL 1=\{Z \colon \, Z' \equivT Z \oplus \Halt \}$ contains every 2-random and hence  is conull (see, for instance,~\cite{DHbook}). Thus,  the following class is also null: 
\begin{equation*}
%\label{H}
\HHH = \left\{Z \,\colon \, \ex f \ltt Z' \,\, [ f\in \PP]\right\}. 
\end{equation*}
This class contains $\Superhigh$ because $\ES''$ truth-table computes a function that is d.n.c.\ relative to $\ES'$. 

Since $\HHH$ is $\SI 3$ and null, we already know, by the result of Hirschfeldt and Miller mentioned in the introduction, that the class $\HHH^\Diamond$ contains a noncomputable set. We now strengthen this fact. 

\begin{theorem} \label{prop:left2right} 
	Every c.e., strongly jump-traceable set is in $\HHH^\Diamond$, that is, is computable from every $1$-random set in $\HHH$.  
\end{theorem}

%\begin{proof}
Fix a truth-table reduction $\Delta$. We will define a benign cost function~$c$ such that for each  set~$A$, and each $1$-random set~$Z$, 
\bc  $\Delta({Z'})$ is $ \twoset$-valued  d.n.c.\ relative to $\Halt$ and $A$ obeys~$c$  
 $\RA$ $A \leT Z$. \ec
Theorem \ref{prop:left2right} then follows from the result from \cite{GreenbergNies:benign}, that every c.e., strongly jump-traceable set obeys every benign cost function. 

\subsection{Discussion}

We first explain in intuitive terms how to obtain the cost function~$c$. The overall strategy has roots in the proof in \cite{CholakDowneyGreenberg} that every c.e., strongly jump-traceable set is not ML-cuppable (see \cite{DHbook} or \cite{Nbook} for a definition of this concept), and in the proof in \cite{GreenbergNies:benign} that every c.e., strongly jump-traceable set is computable from every LR-hard $1$-random set. 

Suppose that we are given a c.e.~set $A$, and we wish to show that $A\leT Z$ for all $1$-random sets $Z$ such that $\Delta(Z')\in \PP$. We implicitly devise a Turing functional that reduces $A$ to such sets $Z$. Since there are uncountably many such sets $Z$, and they are not all definable in any way, we have to work with finite initial segments of such $Z$---equivalently, with clopen classes of such $Z$. We can describe our strategy as a two-pronged attack. First, we require evidence that some clopen class $\CC$ consists of sets $Z$ such that $\Delta(Z')\in \PP$. If we find such evidence, at some stage~$s$, then we decide that the sets in $\CC$ compute some initial segment of $A_s$. Second, if we later discover that this computation is incorrect because $A$ has changed, \emph{and} if it still seems like the sets in  $\CC$ satisfy $\Delta(Z')\in \PP$, then we try to make these sets non-$1$-random. In terms of the Kolmogorov complexity definition of $1$-randomness, essentially what we do is give initial segments of sets in $\CC$ short descriptions; the technical device we actually use is a Solovay test~$\mathcal G$,  which we describe below. (A \emph{Solovay test} is a c.e.\ collection of clopen sets $C_0,C_1,\ldots$ such that $\sum_i \lambda C_i < \infty$. It is easy to check that if a set $X$ is $1$-random and $\mathcal S$ is a Solovay test, then $X$ can be in only finitely many elements of $\mathcal S$; see \cite{DHbook} or \cite{Nbook} for a proof.) Viewed backwards, this derandomization allows us to correct the functional. The cost function $c$ is defined by tracking our beliefs and thus ``pricing'' the changes in the set $A$ according to the amount of correction that would be required, were $A$ to change. 

The combinatorial heart of the argument is the exact designation of when we believe that a clopen class  $\CC$ consists of sets $Z$ such that $\Delta(Z')\in \PP$. This is the basic tension: on the one hand, if indeed $\Delta(Z')\in \PP$, then we need to ensure that we believe this fact for infinitely many initial segments of $Z$. On the other hand, we cannot run wild and issue too many short descriptions: the total weight of those descriptions has to be finite. In other words, by derandomizing strings, we may ask for corrections in the functional, but this right is limited---we cannot ask for too much. If we believe too many strings, the total measure of the Solovay test will not be finite. 

To decide whether to believe a clopen class $\CC$, we define a function $\aaa \leT \Halt$, by giving it a computable approximation $\seq{\aaa_s}$. We believe $\CC$ at stage $s$ if the stage $s$ approximation to $\Delta(Z')$ for sets $Z\in \CC$ differs from $\alpha_s$ on designated locations (or really, from the coding of $\alpha$ in $J^{\Halt}$ at that stage). By designating a large number of such locations, we can ``keep ahead of the game'' by changing $\alpha_s$ if it appears that we believe clopen classes that are too large (in the sense of measure). This prophylactic approach is really the main point of the argument. 

\subsection{The proof of Theorem \ref{prop:left2right}}
 
We now give the details. Let $(I_e)$ be the sequence of consecutive intervals of $\NN$ of length~$e+1$. Thus $\min I_e = e(e+1)/2$. As mentioned, we define a function $\alpha$, partial computable in $\emptyset'$ (which will actually be total). By universality of $J^{\emptyset'}$, and by the recursion theorem, we are given a computable function $p$ that reduces $\alpha$ to $J^{\emptyset'}$: for all $x$, $\alpha(x)\simeq J^{\Halt}\! (p(x))$. 

% At stage~$s$ of the construction, let $x< s$ and  $x \in I_e$.  
% If $p(y) $ is undefined  at stage~$s$  for some  $y \in I_e$,  let $\aaa_s(x) = 0$. Otherwise,  
Let $s<\w$. To define $\alpha_s$, we first let $\CCC_{e,s}$ be the clopen set of oracles~$Z$ such that $\Delta(Z')$ agreed with  $1- \aaa$ on $I_e$  at some stage~$t$ after the last change of $\aaa \uhr {I_e}$. That is, let 
\begin{equation} \label{eqn:CC} \CCC_{e,s}= \{Z\colon \, \exo t_{v \le t \le s} \fa x \in I_e \, [ 1- \aaa_t(x) = \Delta(Z'_t, p(x))]\}, \end{equation}
where $v \le s $ is greatest such that $v =0$ or $\aaa_v \uhrn  {I_e} \neq \aaa_{v-1} \uhrn I_e $.  
  For each $e< s$,   if $\leb \CCC_{e,s-1} \le  \tp{-e+1}$   let $\aaa_s \uhrn{I_e}= \aaa_{s-1} \uhrn{I_e}$.
  Otherwise,    change $\aaa \uhrn { I_e}$:  define  $\aaa_s \uhrn{I_e}$ in such a way that  $\leb \CCC_{e,s} \le \tp{-e}$.  

\begin{claim}  $\aaa(x) = \lim_s \aaa_s(x)$ exists for each~$x$.\end{claim}
 
\begin{proof}[Proof of the Claim]
  We rely on  a measure theoretic fact first  used in a related context  (see \cite[Exercise 1.9.15]{Nbook}). Suppose $n \in \NN$ and we are given measurable classes $\BBB_i $  for $1\le i \le N$,   and  $\leb \BBB_i \ge \tp{-e}$ where $e\in \NN$. If $k \in \NN$ is such that $N  > 2^e k$,  then there is a set $F \sub \{1,\ldots, N\}$ such that $|F| = k+1$ and  $ \bigcap_{i \in F} \BBB_i \neq \ES$.  Beyond proving the claim, this fact   will later  yield  a computable bound in $x$ on the number of changes of $\aaa_s(x)$. 
% For instance, if~$N=5$  classes of measure at least~$1/2$ are given, then   the intersection  of three of them is nonempty.  

 Suppose    that $  v_1 < \cdots < v_N$ are consecutive stages at which $\aaa \uhrn I_e$ changes. 
% Thus $p\uhrn I_e$ is defined from $v_1$ on. 
Note that  for each $i< n$, the measure $\leb \CCC_e$ increases by at least $\tp{-e}$ from stage~$v_i$ to $v_{i+1}$.  Therefore  $\leb \BBB_i \ge \tp{-e}$ for each $i 
 \le N$, where 
\bc $\BBB_i = \{ Z\colon \,  Z'_{v_{i+1}} \uhr k  \neq Z'_{v_i} \uhr k\}$,\ec 
and $k = \use \Delta(\max p(I_e))$. Note that the intersection of any~$k+1$ many of the  $\BBB_i$ is empty. Thus $N \le 2^e k$ by the measure theoretic fact mentioned above. 
\end{proof}

% Since $\aaa \leT \Halt$, by the Recursion Theorem, we can now assume that~$p$ is a  reduction function for $\aaa$ relative to $\Halt$. 

In fact, we have a computable bound~$g$  on the number of changes of $\aaa\uhrn {I_e}$, given by $ g(e) = 2^e \use \Delta (\max p( I_e))$. 
 
  We define a   cost function~$c$ by  $c(x,s) = \tp{-x} $ for each $x\ge s$; if $x<s$,  and~$e \le x$ is
 least such that $e=x$ or 
 $\aaa_s \uhrn{I_e} \neq  \aaa_{s-1} \uhrn{I_e}$,     let 
\bc $c(x,s) =
\max(c(x,s-1), \tp{-e} )$.  \ec
To show that~$c$ is benign, suppose that  $0=v_0 < v_1 < \cdots < v_n$ and $c(v_i, v_{i+1}) \ge \tp{-e}$ for each~$i< n$. Then  $\aaa_s \uhrn{I_e}\neq  \aaa_{s-1} \uhrn{I_e}$   for some~$s$ such that $v_i < s \le v_{i+1}$. Hence $n \le g(e)$. 

To complete the proof of Theorem~\ref{prop:left2right}, let~$A$ be a c.e.\ set that is strongly jump-traceable. By \cite{GreenbergNies:benign}, there is  a computable enumeration $\seq {A_s} \sN s$ of~$A$ that obeys~$c$. 
 
 The rest of the argument actually works for a computable approximation  $\seq{A_s}\sN s$  to a $\DII$ set~$A$.
 We build a  
{Solovay test}  $\mathcal G$   as follows: {when $A_{t-1}(x)\neq A_t(x)$,  we put   $\CCC_{e,t}$ defined in (\ref{eqn:CC}) into $\mathcal G$ where~$e$ is largest such that $\aaa \uhrn {I_e}$ has been stable from~$x$ to~$t$.}  Then $\tp{-e} \le c(x,t)$. Since $\leb \CCC_{e,t} \le \tp{-e+1} \le 2c(x,t)$ and   the computable approximation of~$A$ obeys~$c$, the set $\mathcal G$ is indeed a Solovay test.     

Choose~$s_0$ such that $\sss\not \subseteq Z$ for each  $[\sss]$ enumerated   into~$\mathcal G$ after stage~$s_0$. To show $A \leT Z$, given an input $y\ge s_0$, using~$Z$ as an oracle,  compute~$s>y$ such that $1-\aaa_s(x)= \Delta(Z'_s; x)$ for each $x< y$. Then we claim that $A_s(y)=A(y)$.  Assume not, so that $A_t(y) \neq A_{t-1}(y)$ for some $t>s$,  and let $e \le y$ be largest such that $\aaa\uhrn {I_e}$ has been stable from~$y$ to~$t$. Then by stage $s> y$ the set~$Z$ is in $\CCC_{e,s} \sub \CCC_{e,t}$, so we 
put~$Z$ into $\mathcal G$ at stage~$t$, which is a contradiction.

\

%%%%%%%%%%%%%%%%%%%%%%

\section{$\Superhigh^\Diamond$ is contained in $\SJT$} \label{sec:superhigh2}

In this section we prove the right-to-left direction of Characterization II of the strongly jump-traceable c.e.\ sets: every c.e.\ set that is computable from every superhigh $1$-random set is strongly jump-traceable. 

As in Section \ref{sec:superhigh1}, we prove a somewhat stronger result. For any set $G\subseteq \w$, we replace the class of superhigh sets by the class
\[ \CC_G= \{Y \colon \, G \ltt Y'\}.\]
This class is a subclass of the superhigh sets if $\ES'' \leT G$. No matter what $G$ is, we show that every set in $(\CCC_G)^\Diamond$ is strongly jump-traceable.

\subsection{The path from computable enumerability to superlowness and jump-traceability}

Fix $G\subseteq \w$. We want to prove that every c.e.\ set $A$ that is computable from every $1$-random set in $\CCC_G$ is strongly jump-traceable. As mentioned in Subsection \ref{subsec:ce assumption}, the assumption on $A$ that we actually use, rather than $A$ being c.e., is that $A$ is superlow and jump-traceable:

\begin{theorem}\label{thm:main superhigh thm}
	Let $A$ be a superlow, jump-traceable set, let $G\subseteq \w$, and suppose that for any $1$-random set $Z$ such that $G\ltt Z'$ we have $A\leT Z$. Then $A$ is strongly jump-traceable. 
\end{theorem}

In order to replace \emph{c.e.}\ by \emph{superlow and jump-traceable}, we need the following lemma:

\begin{lemma}\label{lem:ce to SLJT in superhigh}
	Let $G\subseteq \w$. If $A$ is a c.e.\ set that is computable from every $1$-random set $Z$ such that $G\ltt Z'$, then $A$ is superlow and jump-traceable. 
\end{lemma}

Lemma \ref{lem:ce to SLJT in superhigh} follows from Corollary \ref{cor:ce-base-for-randomness} and the following  consequence of    Kjos-Hanssen and Nies \cite[Theorem  3.5]{Kjos.Nies:nd}:

\begin{theorem}\label{thm:preparation superhigh theorem}
	For any $G\subseteq \w$ there is an incomplete $1$-random set $Z$ such that $G\ltt Z'$. 
\end{theorem}

As a gentle introduction to Theorem~\ref{thm:main superhigh thm},  we include  a proof of Theorem~\ref{thm:preparation superhigh theorem}. It turns out that the proof of Theorem \ref{thm:main superhigh thm} is closely related to our proof of Theorem \ref{thm:preparation superhigh theorem},  in a fashion even stronger than the way the proof of Theorem \ref{thm:Pi01 theorem} relates to the proof of the superlow basis theorem. For the proof of Theorem \ref{thm:preparation superhigh theorem}, we start with a $\PPI$ class $\SSS$ consisting of $1$-random sets, enumerate a set $A$, and use a generalized version of \Kuc{} coding to build some $Z\in \SSS$ that codes $G$, in the sense that $G\ltt Z'$, but that avoids $A$, in the sense that $A\nle_\Tur Z$. The proof of Theorem \ref{thm:main superhigh thm} is a reversal, of sorts, of the same situation, in which $A$ is given, but we try to construct such a set $Z\in \SSS$ nonetheless. Our failure to avoid $A$ is then translated, as was done in Section \ref{sec:main}, into a golden pair, and so into an enumeration of a trace for the given $A$-partial computable function. In this way, the proof of Theorem \ref{thm:preparation superhigh theorem} serves as a blueprint for the proof of Theorem \ref{thm:main superhigh thm}. Also, in some sense, this argument shows that strong jump-traceability is exactly the level at which the power of upper-cone avoidance in conjunction with coding fails in ``codable'' $\PPI$ classes such as classes of $1$-random sets or Medvedev complete classes (i.e., $\PPI$ classes that have the highest possible degree in the Medvedev lattice of mass problems). 
 %Reference to Anna Lydia Plurabelle by Joyce needed here.

We remark that for $G = \emptyset''$, which is the case we are interested in to prove the right-to-left direction of Characterization II, Theorem \ref{thm:preparation superhigh theorem} can be proved by a $1$-random pseudo-jump inversion technique, to obtain a $\DII$ $1$-random set $Z$. See \cite[Theorem 6.3.14]{Nbook}.

\subsection{\Kuc{} coding} \label{subsec:Kuc-coding}

We start with a review of \Kuc{} coding into $\PPI$ classes of $1$-random sets.
For a string $\tau$ and a class $\BBB\subseteq 2^\w$, let $\BBB \mid \tau = \left\{ X\in 2^\w\,:\, \tau X\in \BBB\right\}$. If $\BBB$ is a measurable class, then $\leb(\BBB \mid \tau) = 2^{|\tau|}\leb (\BBB\cap [\tau])$. 

Recall that a string $\tau\in 2^{<\w}$ is called \emph{extendible} in a $\PPI$ class $\PP$ if $\PP\cap [\tau]$ is nonempty, or equivalently, if $\PP \mid \tau$ is nonempty. If $\PP$ has positive measure, this notion can be strengthened: for any $r<\w$, we say that $\tau\in 2^{<\w}$ is \emph{$r$-extendible} in $\PP$ if $\leb(\PP \mid \tau) \ge 2^{-r}$. 

Let $\PP$ be a $\PPI$ class of positive measure, and let $r<\w$ be sufficiently large so that $\leb \PP\ge 2^{-r}$. We define an embedding of the full binary tree into subclasses of $\PP$ of positive measure defined as the intersections of $\PP$ with basic clopen classes. That is, for every finite binary string $\aaa$ we define a string $\kuc_r(\PP,\aaa)$ such that:
\begin{itemize}
	\item If $\alpha\subset \beta$, then $\kuc_r(\PP,\aaa)\subset \kuc_r(\PP,\beta)$; if $\alpha\perp \beta$, then $\kuc_r(\PP,\alpha)\perp \kuc_r(\PP,\beta)$. 
	\item For all $\alpha\in 2^{<\w}$, the string $\kuc_r(\PP,\aaa)$ is $r+|\alpha|$-extendible in $\PP$. 
\end{itemize}

The definition of $\kuc_r(\PP,\aaa)$ is done recursively in $\aaa$, based on the following lemma. 

\begin{lemma}[\Kuc; see also {\cite[Lemma 3.3.1]{Nbook}}] \label{bk:extensions} 
	Suppose that~$\PP$ is a $\PPI$ class, $l<\w$, and $\tau\in 2^{<\w}$ is $l$-extendible in $\PP$. Then there are at least two strings $\s \supset\tau$ of length $|\tau| +l+1$ that are $l+1$-extendible in $\PP$. 
\end{lemma}

We let $\kuc_r(\PP,\ES)$ be the leftmost string $\tau$ of length $r$ that is $r$-extendible in $\PP$. If $\kuc_r(\PP,\aaa)$ has been defined, then we let $\kuc_r(\PP,\aaa0)$ be the leftmost extension of $\kuc_r(\PP,\aaa)$ of length $|\kuc_r(\PP,\aaa)|+r+|\aaa|+1$ that is $r+|\aaa|+1$-extendible in $\PP$, and let $\kuc_r(\PP,\aaa1)$ be the rightmost such extension. 

For all $\aaa$, the length of $\kuc_r(\PP,\aaa)$ is 
\[ r + |\aaa|r + \binom{|\aaa|}{2} - 1. \]
We define 
\[ \ell(n,r) = r(n+1) + \binom{n}{2} -1 ;\]
so $|\kuc_r(\PP,\aaa)|=\ell(|\aaa|,r)$ for all $\aaa$. The point is that the map $\aaa\mapsto \kuc_r(\PP,\aaa)$ is not computable, but $\ell$ is. 

This simple version of \Kuc{} coding is sufficient to prove the \Kuc-G\'{a}cs Theorem, that every set is computable from a $1$-random set. For let $\PP$ be a $\PPI$ class consisting of $1$-random sets. We know that $\leb\PP>0$, so fix some $r$ such that $\leb\PP \ge 2^{-r}$. Let $G\in 2^\w$ and let 
\[ Z = \bigcup_n \kuc_r(\PP,G\uhr n) .\]
The reason that $G\leT Z$ is that we can effectively determine, given $Z\uhr{\ell(n,r)} = \kuc_r(\PP,G\uhr n)$, whether $Z\uhr{\ell(n+1,r)}$ is the leftmost or rightmost extension of $Z\uhr{\ell(n,r)}$ of its length that is $r+n+1$-extendible in $\PP$, because the set of $l$-extendible strings is co-c.e., uniformly in $l$. 

This last argument points to an effective approximation to the coding strings. Recall the descending approximation $\seq{\PP_t}$ to $\PP$ by clopen sets (from Subsection \ref{subsec:notation}). Of course, $\leb \PP_t \ge \leb \PP$ for all $t$, so if $\leb \PP\ge 2^{-r}$, then for all $t$ and all $\aaa$, the string $\kuc_r(\PP_t,\aaa)$ is defined, and indeed effectively obtained, uniformly in $\aaa$ and $t$ (and in $r$). In fact, $\seq{\kuc_r(\PP_t,\aaa)}_{t<\w}$ is a computable approximation, with a computably bounded number of changes,  to the function $\aaa\mapsto \kuc_r(\PP,\aaa)$. For,  if $\kuc_r(\PP_t,\aaa)$ is stable along an interval of stages, then in this interval we see at most $2^{\ell(|\aaa|+1,r)- \ell(|\aaa|,r)}$ many changes in $\kuc_r(\PP_t,\aaa0)$ (and the same holds for $\kuc_r(\PP,\aaa1)$), again because the set of $r+|\aaa|+1$-extendible extensions of $\kuc_r(\PP_t,\aaa)$ of length $\ell(|\aaa|+1,r)$ is (uniformly) co-c.e. Inductively, we obtain the following:

\begin{lemma}\label{lem:Kucera approx}
	For any $\aaa$, the number of stages $t$ such that 
	\[ \kuc_r(\PP_{t+1},\aaa) \ne \kuc_r(\PP_{t},\aaa) \]
	is bounded by $2^{\ell(|\aaa|,r)}$. 
\end{lemma}

\subsection{Lower bound  functions} 
If we want to combine coding with Friedberg-Muchnik style diagonalization, one $\PPI$ class is not sufficient: we need to pass to $\PPI$ subclasses that avoid computations that currently look correct. To use \Kuc{} coding on each of these classes, we need, effectively in the index of a class, a positive lower bound on its measure. The lower bound  function is the map giving this lower bound. 

In both proofs, of Theorem \ref{thm:preparation superhigh theorem} and of Theorem \ref{thm:main superhigh thm}, we being with a $\PPI$ class $\SSS$ of $1$-random sets and enumerate a c.e.\ set $V$; from each $v\in V$ we effectively compute an index for a $\PPI$ subclass $\PP^{(v)}$ of $\SSS$. A  \emph{lower bound  function} for   $\seq{\PP^{(v)}}_{v \in V}$  is a (total) computable function $q$ such that for all $v\in V$, if $\PP^{(v)}$ is nonempty, then $\leb\PP^{(v)}\ge 2^{-q(v)}$. 

\begin{lemma} \label{lem:KucLB} 
Let $\SSS$ be a $\PPI$ class of $1$-random sets.
	Any effective list $\seq{\PP^{(v)}}_{v \in V}$ of $\PPI$ subclasses of $\SSS$ has  a lower bound  function. Moreover, a computable index for the  function can be computed effectively from an index for the enumeration $\seq{\PP^{(v)}}$. 
\end{lemma}

\begin{proof}
This result is best proved using a basic result on prefix-free Kolmogorov complexity:
By the Kraft-Chaitin Theorem (see for instance \cite[Theorem 2.2.17]{Nbook}), there is a coding constant~$c_0$ such that $ \leb \PP^{(v)} \le 2^{-K(v)-c_0} \rightarrow \PP^{(v)} = \ES$ (see \cite[Exercise 3.3.3 and its solution]{Nbook}). Fix $d \in \w$ such that $K(v) \le 2 \log v +d$. Let $q(v) = 2 \log v +c_0+d$. The constant $c_0$ can be obtained effectively from the enumeration $\seq{\PP^{(v)}}$ because the Kraft-Chaitin Theorem is uniform. 
\end{proof}

Hence, by the recursion theorem, we may assume that a lower bound function for the classes enumerated during the construction is known to us during the construction; we fix such a function $q$. 

Since the descending, clopen, effective approximation $\seq{\PP_t}$ to a $\PPI$ class $\PP$ is obtained effectively from a canonical index for $\PP$, we get such an approximation $\seq{\PP^{(v)}_t}$ uniformly for all $v\in V$. We may assume that for all $v\in V$ and $t<\w$, if $\leb\PP^{(v)}_t < 2^{-q(v)}$, then $\PP^{(v)}_t$ is empty. To omit an index, for all $v\in V$ such that $\PP^{(v)}$ is nonempty, we let $\kuc(\PP^{(v)},\aaa) = \kuc_{q(v)}(\PP^{(v)},\aaa)$, and similarly, $\kuc(\PP^{(v)}_t,\aaa) = \kuc_{q(v)}(\PP^{(v)}_t,\aaa)$, which is defined if and only if $\PP^{(v)}_t$ is nonempty, a condition that is effectively detectable. The map $(v,t,\aaa)\mapsto \kuc(\PP^{(v)}_t,\aaa)$ is computable on its domain, which is itself computable. 
 
% \begin{itemize}
% 	\item We assume that every $v\in V$ is of the form $(u,\tau)$, where $\tau\in 2^{<\w}$ is a string such that $\PP^{(v)}\subseteq [\tau]$. We may assume that for all $t$, $\PP^{(v)}_t\subseteq [\tau]$. 
% \end{itemize}

\subsection{Proof of Theorem \ref{thm:preparation superhigh theorem}}

We start with a $\PPI$ class of $1$-random sets $\SSS$; this class has positive measure. We enumerate a c.e.\ set $A$, against which we try to diagonalize. For coding, we approximate coding strings $\sigma_\gamma$ by giving their stage $s$ versions $\sigma_{\gamma,s}$. For diagonalization, we approximate $\PPI$ classes $\SSS_\gamma$, subclasses of $\SSS$, such that for all $\gamma$ and all $Z\in \SSS_\gamma$, we have $\Phi_{|\gamma|}(Z)\ne A$. For compatibility of coding and diagonalization, we ensure that for all $\gamma$ we have $\SSS_\gamma\subset [\sigma_\gamma]$, and that for both $j<2$, the string $\sigma_{\gamma j}$ is extendible in $\SSS_\gamma$. 

At stage $s$ we define $\SSS_{\gamma,s}$, an approximation to $\SSS_\gamma$. As mentioned above, we will enumerate a c.e.\ set $V$. Each class $\SSS_{\gamma,s}$ will be of the form $\PP^{(v)}$ for some $v\in V$. If $\SSS_{\gamma,s-1} = \PP^{(v)}$, and at stage $s$ we decide not to change this class, that is, we decide that $\SSS_{\gamma,s} = \SSS_{\gamma,s-1}$, then of course we have $\SSS_{\gamma,s} = \PP^{(v)}$ for the same $v$. If we decide to pick a new class, so that $\SSS_{\gamma,s}\ne \SSS_{\gamma,s-1}$, then we enumerate a new element $u$ into $V$ at stage $s$, and define $\SSS_{\gamma,s} = \PP^{(u)}$. The number $u$ will equal $\seq{\gamma,k}$, where $k$ is the number of previous versions of $\SSS_{\gamma,t}$, and we identify $2^{<\omega}$ with $\omega$ in a natural way.

% Our promise that  $\SSS_{\gamma,s}\subset[\sigma_{\gamma,s}]$ is ensured by enumerating $\sigma_{\gamma,s}$ as the second coordinate of $v$. From now, this will be done implicitly. Note that the promise $\SSS_{\gamma,s}\subset[\sigma_{\gamma,s}]$ implies that a new value for $\SSS_{\gamma,s}$ is picked each time $\sigma_{\gamma,s}$ changes.

Note the multiplicity of the subscript $s$: $\SSS_{\gamma,s}$ is a $\Pi^0_1$ class; the clopen class that is its stage $s$ approximation will be denoted by $(\SSS_{\gamma,s})_s$.

\subsubsection*{Construction}

At stage $s$, we define $\SSS_{\gamma,s}$ and $\sigma_{\gamma,s}$ by recursion on $\gamma$. Starting with  $\gamma = \estring$, we let $\sigma_{\estring,s} = \estring$. 

Now suppose that $\sigma_{\gamma,s}$ is defined. If $\gamma\ne \estring$, let $\gamma^- = \gamma\uhr{|\gamma|-1}$ be $\gamma$ with the last bit chopped off. By induction, $\SSS_{\gamma^-,s}$ is already defined. If $\gamma = \estring$, let $\SSS_{\gamma^-,s}= \SSS$.

There are three possibilities:

\begin{enumerate}
	\item If $s=0$, or if $s>0$ but $\sigma_{\gamma,s}\ne \sigma_{\gamma,s-1}$, then we pick a new value for $\SSS_{\gamma,s}$. Let $v$ be the new index for $\SSS_{\gamma,s}$, which we enumerate into $V$.
	
We let
	\[ \SSS_{\gamma,s} = \left\{ Z\in \SSS_{\gamma^-,s}\cap [\sigma_{\gamma,s}]\,:\, \lnot \,(\Phi_{|\gamma|}(Z,v)\converge = 0) \right\} .\]
	
	\item If $s>0$ and $\sigma_{\gamma,s} = \sigma_{\gamma,s-1}$, but $(\SSS_{\gamma,s-1})_s$ is empty, then let $\SSS_{\gamma,s} = \SSS_{\gamma^-,s}\cap [\sigma_{\gamma,s}]$. 
	
	\item Otherwise, let $\SSS_{\gamma,s} = \SSS_{\gamma,s-1}$. 
\end{enumerate}

After $\SSS_{\gamma,s}$ is defined, for both $j<2$, we let $\sigma_{\gamma j,s} = \kuc(\SSS_{\gamma,s}, \seq{j})$.

\subsubsection*{Verification}

\begin{claim} \label{clm:protosuperhigh-limit} 
		For all $\gamma$, both $\sigma_{\gamma,s}$ and $\SSS_{\gamma,s}$ stabilize to final values $\sigma_{\gamma}$ and $\SSS_{\gamma}$. The approximations are both $\w$-c.e. 
\end{claim}

\begin{proof}
	By induction on $\gamma$. We always have $\sigma_{\estring,0} = \estring$. 
	
	Let $\gamma$ be any string, and suppose that in an interval $[t_0,t_1]$ of stages, the value of $\sigma_{\gamma,s}$ is constant. Then $\SSS_{\gamma,s}$ is changed at most once between stages $t_0$ and $t_1$. 
	
	Suppose now that in an interval $[t_0,t_1]$ of stages, the class $\SSS_{\gamma,s}$ is not redefined. 	
Let $v\in V$ be the number such that $\SSS_{\gamma,s} = \PP^{(v)}$ for all $s\in [t_0,t_1]$. Let $j<2$. By Lemma \ref{lem:Kucera approx}, in stages between $t_0$ and $t_1$, the value of $\sigma_{\gamma j,s}$ changes at most $2^{\ell(1,q(v))}$ many times. 

	Hence both $\sigma_{\gamma,s}$ and $\SSS_{\gamma,s}$ reach a limit. To see that the number of changes is bounded computably in $\gamma$, we again argue by recursion on $\gamma$. If $\SSS_{\gamma,s}$ changes at most $m$ many times, then we recall that the $V$-indices for $\SSS_{\gamma}$ are $(\gamma,0), (\gamma,1), \dots$, and so for both $j<2$, the number of times $\sigma_{\gamma j,s}$ changes is bounded by 
	\[ \sum_{k\le m} 2^{\ell(1,q(\seq{\gamma,k}))}, \]
	which is computable. 	
\end{proof}

It is clear from the instructions that for all $\gamma$, we have $\SSS_{\gamma}\subset \SSS_{\gamma^-}$ and $\SSS_\gamma\subset [\sigma_\gamma]$.

\begin{claim} \label{clm:protosuperhigh-extendible}
	For all $\gamma$, we have $\SSS_{\gamma^-}\cap [\sigma_\gamma]\ne \ES$. 
\end{claim}

\begin{proof}
	By induction on $\gamma$. For $\gamma=\estring$ the claim follows from $\sigma_\estring=\estring$ and $\SSS\ne \ES$. 
	
	Suppose that $\SSS_{\gamma^-}\cap [\sigma_{\gamma^-}]\ne \ES$. Let $j<2$. We show that $\SSS_{\gamma}\cap [\sigma_{\gamma j}]\ne \ES$. 
	
	First, we see that $\SSS_\gamma\ne \ES$, which follows from the instructions. Let $s_0$ be the stage at which $\sigma_{\gamma,s}$ stabilizes. At stage $s_0$, we pick a new value for $\SSS_{\gamma}$. If at a later stage $s_1$ we discover that $\SSS_{\gamma,s_0}$ is empty, then we switch to the final value $\SSS_{\gamma} = \SSS_{\gamma^-}\cap [\sigma_\gamma]$, which by induction is nonempty. Otherwise, $\SSS_\gamma = \SSS_{\gamma,s_0}$ is nonempty. 
	
	Now that we know that $\SSS_\gamma$ is nonempty, we know that each \Kuc{} string $\kuc(\SSS_\gamma,\aaa)$ is extendible in $\SSS_\gamma$, in particular $\sigma_{\gamma j} = \kuc(\SSS_\gamma, \seq{j})$.
\end{proof}

Now define a c.e.\ set $A$ as follows: at a stage $s>0$, if $\sigma_{\gamma,s} = \sigma_{\gamma,s-1}$ but $(\SSS_{\gamma,s-1})_s = \ES$ (that is, case (2) of the construction holds for $\gamma$ at stage $s$), then enumerate the $V$-index of $\SSS_{\gamma,s-1}$ (that is, the number $v\in V$ such that $\SSS_{\gamma,s-1} = \PP^{(v)}$) into $A$. 

\begin{claim} \label{clm:protosuperhigh-diagnoalize}
	For all $\gamma$ and all $Z\in \SSS_\gamma$, we have $\Phi_{|\gamma|}(Z)\ne A$. 
\end{claim}

\begin{proof}
	Fix $\gamma$, and let $s_0$ be the stage at which $\sigma_\gamma$ stabilizes. Let $v$ be the $V$-index of $\SSS_{\gamma,s_0}$. 	
	
	If there are no changes in $\SSS_{\gamma,s}$ after stage $s_0$, that is, if $\SSS_{\gamma} = \SSS_{\gamma,s_0}$, then $v\notin A$, and by the definition of $\SSS_{\gamma,s_0}$, for no $Z\in \SSS_\gamma$ do we have $\Phi_{|\gamma|}(Z,v)= 0$. 
	
	Otherwise, at some stage $s_1>s_0$ we redefine $\SSS_{\gamma,s_1} = \SSS_{\gamma^-}\cap [\sigma_\gamma]$, and there are no further changes in $\SSS_\gamma$. But this redefinition is done only because at stage $s_1$ we discover that $\SSS_{\gamma,s_0}= \ES$. By the definition of $\SSS_{\gamma,s_0}$, we thus have $\Phi_{|\gamma|}(Z,v)=0$ for all $Z\in \SSS_{\gamma}$. But in this case, $v\in A$. 
\end{proof}

Let $G\in 2^\w$. Define $Z = \bigcup_{n} \sigma_{G\uhr n}$. If $n\le m<\w$ then by Claim \ref{clm:protosuperhigh-extendible},
\[ \ES \ne \SSS_{G\uhr{m}} \cap \left[\sigma_{G\uhr{m+1}}\right] \subseteq \SSS_{G\uhr{n}}\cap \left[\sigma_{G\uhr{m+1}}\right],\]
so by compactness, $Z\in \SSS_{G\uhr n}$ for all $n$. By Claim \ref{clm:protosuperhigh-diagnoalize}, $A\nle_\Tur Z$, so $Z$ is incomplete. 

\begin{claim} \label{clm:protosuperhigh-coding}
	$G\ltt Z'$. 
\end{claim}

\begin{proof}
	By Fact \ref{fact:BLR-and-TT} (due to Cole and Simpson), it is equivalent to show that $G\in \BLR(Z)$. 
	
To construct a stage $s$-approximation to $G$, note that if $\gamma$ and $\delta$ are incomparable strings, then $\sigma_{\gamma,s}$ and $\sigma_{\delta,s}$ are also incomparable. Hence
\[ G_s = \bigcup \left\{ \gamma\,:\, \sigma_{\gamma,s}\subset Z \right\} \]
is well-defined (it may be finite). It is (uniformly) $Z$-computable, because for all $\gamma$ and $s$, we have $|\sigma_{\gamma,s}|\ge |\gamma|$. 

Let $x<\w$. By Claim \ref{clm:protosuperhigh-limit}, there is a computable bound on the number of stages at which any of the strings $\sigma_{\gamma,s}$ for any string $\gamma$ of length $x+1$ may change. The approximation $G_s(x)$ may change only at such stages. Hence $\seq{G_s}$ witnesses that $G\in \BLR(Z)$. 
\end{proof}

\subsection{Discussion of the proof of Theorem \ref{thm:main superhigh thm}}

The rest of this section is dedicated to the proof of Theorem~\ref{thm:main superhigh thm}. Fix $G\in 2^\w$. We assume that $A$ is a superlow and jump-traceable set, computable from every $1$-random set $Z$ such that $G\ltt Z'$. As in Section~\ref{sec:main}, fix an order function $h$, and an $A$-partial computable function $\theta$. By Theorem~\ref{thm: JT and SL}, we can fix a restrained $A$-approximation $\seq{A_s,\Gamma_s}$ to $\theta$, witnessed by a computable function $g$. 

As discussed above, this proof will follow the idea of the proof of Theorem \ref{thm:Pi01 theorem}, but mirroring the proof of Theorem \ref{thm:preparation superhigh theorem}. Thus we again begin with a nonempty $\Pi^0_1$ class $\SSS$ that contains only $1$-random sets. We will adapt the definition of a golden pair to the current setting; such a pair will arise from a failure to construct some $Z\in \SSS$ that both codes $G$ and does not compute $A$. Again we will approximate strings $\sigma_\gamma$ that are extendible in $\SSS$ and serve as coding strings, and again we will make the approximation $\seq{\sigma_{\gamma,s}}$ an $\w$-c.e.\ one. As in the proof of Theorem \ref{thm:Pi01 theorem}, subclasses in which we attempt to diagonalize against $A$ will be tied to computations $\Gamma^A(x)[s]$ that are under a process of verification. 

The fact that the proof of the superlow basis theorem is linear in nature, but the proof of the coding theorem \ref{thm:preparation superhigh theorem} is not, makes the structure of runs of procedures more complicated in the current proof. Rather than having a linear structure, we now have an (infinitely branching) tree of runs of procedures at any given stage. Thus, while before we had at most one procedure of each type per level $e$, now many of them run in parallel. When a procedure $R^e$ is called, an initial segment $\eta \sub G$ will have been coded into $Z'$ already. So we now have versions $R^{e,\eta}$ for various strings $\eta$. A subprocedure $S^e_x$ has to live with the coding into $Z'$ of a further string of length $h(x)$. Thus, we have versions $S^{e, \eta \aaa}_x$ for each $\aaa$ of length $h(x)$. During the construction, this feature leads to some extra cancellations, as we need to be able to replace a run $S^{e, \eta \beta}_y$ by $S^{e, \eta \aaa}_x$ for $\aaa \subset \beta$. In the definition of golden pairs we will fix an $\eta$ and refer only to runs $S^{e, \eta \aaa}_x$ where $\eta \aaa \sub G$. 

For coding, we again work with an effective list $\PP^{(v)}$ of subclasses of $\SSS$ that we enumerate. Every class we define will be on this list. We mostly leave this as an implicit part of the construction.

\subsection{Golden pairs}
\nopagebreak[4]

\
\nopagebreak[4]

\nopagebreak[4]
\nicebox{ 
\begin{definition}
\label{def:golden pair 2} 

%%%%%%%%%%%
A pair $\QQ,\Phi$, consisting of a nonempty $\Pi^0_1$ class $\QQ= \PP^{(v)}$ and a Turing functional $\Phi$, is a \emph{golden pair} for $\Gamma$, $h$, and $G$, with parameter $\eta \subset G$, if for almost all $x$ such that $\Gamma^A(x)\converge$, with use $u$, for the $\aaa$ of length $h(x)$ such that $\eta \aaa \subset G$, if $X\in \QQ \cap [ \kuc(\QQ, \aaa)]$, then $\Phi(X)\supseteq A\uhr{u}$. 
\end{definition}

%%%%%%%%%%%%%
}
\vsps

\begin{proposition}
\label{prop:golden pairs do the trick 2} If there is a golden pair for $\Gamma$, $h$, and $G$, with parameter $\eta \subset G$, then $\theta = \Gamma^A$ has a c.e.\ trace $\seq {V_x}\sN x$ such that $| V_x| \le 2^{h^2(x)}$ for all $x$. 
\end{proposition}
\begin{proof}

%[Proof of Proposition \ref{prop:golden pairs do the trick 2}] 
Let $\QQ=\PP^{(v)}$ and $\Phi$ be such a golden pair. At stage $s$ we enumerate a number $y$ into $V_x$ if there is a string $\aaa$ of length $h(x)$ such that at that stage we discover that there is a binary sequence $\tau$ for which $\Gamma_s(\tau,x) =y$ and $\Phi(X)$ extends $\tau$ for every $X\in \QQ_s \cap [\kuc(\QQ_s, \aaa)]$. 

To establish the bound on $|V_x|$, let $n= h(x)$. By Lemma \ref{lem:Kucera approx}, for all strings $\aaa$ of length $n$, the string $\kuc(\QQ, \aaa)$ changes at most $2^{\ell(n,q(v))}$ many times. We have $\ell(n,q(v)) \le  (3/4) n^2$ for almost all $n$. Taking the union over all strings $\aaa$ of length $n$, we obtain
$ |V_x| \le 2^n 2^{\ell(n,q(v))}$, which is bounded by $2^{n^2}$ for almost all $n$. 

To establish tracing, if $x\in \dom \theta$, let $\aaa $ be of length $h(x)$ such that $\eta \aaa \sub G$. Let $t_x$ be so large that $\kuc(\QQ, \aaa) = \kuc(\QQ_{s}, \aaa)$ for all $s\ge t_x$. Since $\QQ,\Phi$ is golden, for almost all $x$ and for large enough $s \ge t_x$, we can see that for every $X\in \QQ \cap \kuc(\QQ, \aaa) $ we have $\Phi(X)\supseteq \tau = A\uhr{u}$, where $u$ is the use of $\Gamma^A(x)$. Thus $\theta(x)\in V_x$ for almost all $x$ in the domain of $\theta$. Mending the sequence on finitely many inputs yields a trace as required. 
\end{proof}

Thus, the rest of the proof is devoted to showing that a golden pair exists:

\begin{proposition}
\label{prop:golden pairs exist 2} There is a $\Pi^0_1$ class $\QQ\subseteq \SSS$ and a Turing functional $\Phi$ such that $\QQ,\Phi$ is a golden pair for $\Gamma$, $h$, and $G$. 
\end{proposition}

\subsection{The procedures and the construction} 

The procedure calling structure is now
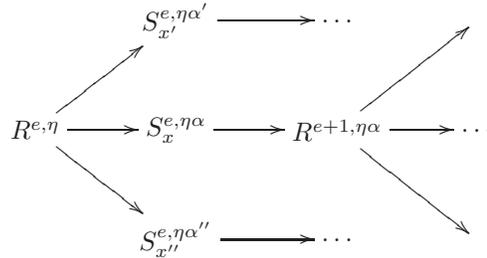
\[ \xymatrix{ & S^{e, \eta \aaa'}_{x'} \ar[r] & \ldots & & \\
R^{e, \eta} \ar[ur] \ar[dr] \ar[r] & S^{e, \eta \aaa}_x \ar[r] & R^{e+1, \eta \aaa} \ar[ur] \ar[dr] \ar[r] & \ldots \\
&S^{e, \eta \aaa''}_{x''} \ar[r] & \ldots & & } \]

At each stage, for each level $e$ and each string $\eta$ such that $R^{e, \eta} $ is running, the strings $\aaa$ such that some $S^{\eta \aaa}_x$ is running and has not returned form a prefix-free set. 

\n \emph{Procedure $R^{e,\eta}$.} This procedure runs with input $\PP^{e,\eta}$ (again, a $\PPI$ class of the form $\PP^{(v)}$ for some $v\in V$), and a parameter $n$. During its run, for every string $\aaa$ of length greater than $n$ and every $x$ such that $|\aaa| = h(x)$, the pair $(\aaa,x)$ is marked either \emph{fresh} or \emph{confirmed}. Initially, all such pairs are fresh. A string $\aaa$ such that $|\aaa|>n$ \emph{requires attention} at stage~$s$ if there is some $x$ such that $|\aaa| = h(x)$ and such that:
\bi 
\item[-] no procedure $S^{e, \eta \beta}_y$ is currently running for any $y<x$ and $\beta \subseteq \aaa$; 
\item[-] $\Gamma^{A}(x)\converge\![s]$; and
\item[-] $(\alpha,x)$ is currently fresh.
\ei
For any string $\aaa$ that requires attention at stage $s$ and is minimal among such strings under the prefix relation (that is, no proper initial segment of $\aaa$ also requires attention), we choose $x$ to be the least number that witnesses that $\aaa$ requires attention, and call a run of the procedure $S^{e,\eta\aaa}_x$ with input $A_s\uhr u$, where $u$ is the use of the computation $\Gamma^A(x)[s]$. We also cancel any run of any procedure $S^{e,\eta \nu}_y$ where $y> x $ and $\aaa \subseteq \nu$. This concludes the instructions for~$R^{e,\eta}$. 

\

For $\PP= \PP^{(v)}$ and a string $\aaa$ let \bc $\PP \seq \aaa = \PP \cap [\kuc (\PP, \aaa)]$. \ec We also let $\PP_s \seq \aaa = \PP_s \cap [\kuc (\PP_s, \aaa)]$ at stage $s$. 

\vsps

\n \emph{Procedure $S^{e, \eta \aaa}_x$.} This procedure is provided with a string $\tau\subset A_s$ such that $\Gamma_s(\tau,x)\converge$, and acts as follows. 
\begin{enumerate}

\item[(a)] Start a run of $R^{e+1, \gamma}$, where $\gamma = \eta \aaa$, with input 
\begin{equation}
\label{eqn:defP} \PP^{e+1, \gamma} = \left\{ X\in \PP^{e, \eta} \seq {\aaa}\,:\, \Phi_e(X)\nsupseteq \tau \right\}
\end{equation}
and parameter $h(x)$. If this is the $i\tth$ run of a procedure of this type $S^{e, \eta \aaa}_x$ ($i \ge 1$), enumerate $v= \la e,x,i\ra$ into $V$ and let $\PP^{(v)} = \PP^{e+1, \gamma}$.

\item [(b)] If at some stage $\PP^{e+1, \gamma}$ becomes empty, cancel the run of $R^{e+1, \gamma}$ and return, marking $(\aaa,x)$ as confirmed. 
\end{enumerate}

As before, a run of $S^{e,\eta\aaa}_x$ with input $\tau$ believes that $\tau\subset A$ and that the current version of $\PP^{e,\eta}\seq{\aaa}$ is correct. Thus, suppose that $S^{e,\eta \aaa}_x$ is called (with input $\tau$) at stage $s$, and $t>s$ is the least stage at which either $\tau\not\subset A_t$ or $\kuc(\PP^{e,\eta}_t,\aaa)\ne \kuc(\PP^{e,\eta}_s,\aaa)$. If the same run of $S^{e,\eta \aaa}_x$ is still running at stage $t$, then it is immediately cancelled (together with all of its subprocedures); otherwise, $(\aaa,x)$ is re-marked as fresh. 

This concludes the instructions for $S^{e, \eta \aaa}_x$. 

\

The construction is started by calling $R^{0, \estring}$ with input $\PP^{(0)} = \SSS$. (Recall that $\SSS$ is a nonempty $\PPI$ class containing only $1$-random sets, defined at the beginning of this proof.)

%%%%%%%%%%%%%%%%%%%%%%%%%%%%%%%%%%%VERIFICATION
\subsection{Verification} 

We show that there is an $e$ and $\eta \subset G$ such that $\PP^{e, \eta},\Phi_e$ is a golden pair for $\Gamma$, $h$, and $G$ for the final version of $\PP^{e, \eta}$. First we do the necessary counting of how often procedures can be called. We begin with the analog of Claim \ref{claim: local counting}. The situation is more complicated here because the number of cancellations of a run $S^{\eta \aaa}_x$ depends on the length of the coding string $\kuc(\PP^{e,\eta},\aaa)$, and hence on the lower bound on the measure of the $\PPI$ class this runs works in. Recall the computable function $q$ from Lemma~\ref{lem:KucLB}.

%%%%%%%%%%%%%%%%CLAIM 1
\begin{claim} \label{claim: local counting 2} 
	There is a computable function $B(x,r)$ such that a run $R^{e, \eta}$, with input $\PP^{e,\eta}= \PP^{(v)}$, calls at most $B(x, q(v))$ many runs of any $S^{e,\eta \aaa}_x$. 
\end{claim}

\begin{proof}
	Fix $x<\w$ and a string $\aaa$ of length $h(x)$. Any call of $S^{e,\eta \aaa}_x$ by the run of $R^{e,\eta}$ beyond the first one is done because a previous run was cancelled, or because a later change caused $(\aaa,x)$ to be re-marked as fresh. This situation can have one of three causes:
	\begin{enumerate}
		\item[(i)] the previous run of $S^{e,\eta \aaa}_x$ had input $\tau$, and later we saw that $\tau\not\subset A_t$;
		\item[(ii)] $\kuc(\PP^{e,\eta}_s, \aaa)$ has changed; 
		\item[(iii)] some run $S^{e, \eta \beta}_y$ was started, where $y< x$ and $\beta \sub \aaa$. 
	\end{enumerate}
The number of times (i) occurs is bounded by $g(x)$; recall that $g$ witnesses that $\seq{A_s,\Gamma_s}$ is a restrained approximation.  The number of times (ii) can occur is bounded by $2^{\ell(|\aaa|, q(v))}$ by Lemma \ref{lem:Kucera approx}. Let $\tilde B(0,r) = 1+ g(0)+ 2^{\ell(h(0),r)}$, and for $x>0$, let $\tilde B(x,r) = 1+ g(x) + 2^{\ell(h(x),r)} + \sum_{y<x}\tilde B(y,r)$. Then $\tilde B(x,q(v))$ bounds the number of calls of $S^{e,\eta\aaa}_x$ by the single run of $R^{e,\eta}$. Summing over all strings $\aaa$ of length $h(x)$, we see that $B(x) = 2^{h(x)} \tilde B(x,r)$ is a bound as required. 
\end{proof}

%%%%%%%%%%%%%%%%CLAIM 2
We proceed to a fact similar to Claim \ref{claim: global counting}. 
\begin{claim}
\label{claim: global counting 2} 
\
\begin{enumerate}
	\item There is a computable bound $M(e,x)$ on the number of calls of any procedure of the form $S^{e,\gamma}_x$.
	\item There is a computable bound $r(e,x)$ on the $V$-index of any class of the form $\PP^{e+1,\gamma}$ that is called by a procedure $S^{e,\gamma}_x$. Hence, with the aid of the function $q$, we get a computable lower bound on the measure of all such classes. 
\end{enumerate}
\end{claim}

%%%%%%%%%
\begin{proof}
	Both parts are computed simultaneously by recursion on $e$.

(2) for $e,x$ follows from (1) for the same pair $e,x$. The $i\tth$ call of any $S^{e,\gamma}_x$ provides its run $R^{e+1,\gamma}$ with input $\PP^{e+1,\gamma} = \PP^{(v)}$, where $v = \seq{e,x,i}$. Thus we can let $r(e,x) = \max_{i\le M(e,x)} {\seq{e,x,i}}$. 

For $e=0$, there is a single run of $R^{0,\estring}$ that is never cancelled, with input $\PP^{0,\estring} = \PP^{(0)}$. So $M(0,x)=B(x,q(0))$ is an upper bound as desired. 

Now assume that $e>0$ and that (1) and (2) have been computed for all pairs $(e',x')$ that lexicographically precede the pair $(e,x)$. By Claim~\ref{claim: local counting 2}, we may let $M(e,x)$ be the product of 
\begin{itemize}
	\item[(i)] a bound on the number of runs of $R^{e, \eta}$ that are called by some $S^{e-1, \eta}_y$ with parameter $h(y)<h(x)$, and 
	\item[(ii)] a bound on the number of times a single run $R^{e, \eta}$ can call $S^{e,\gamma}_x$.
\end{itemize}
Both bounds are obtained by the fact that every run of some $R^{e,\eta}$ that calls some $S^{e,\gamma}_x$ is in turn called by a run of $S^{e,\eta}_y$ for some $y$ such that $h(y)<h(x)$, as $R^{e,\eta}$'s parameter is $h(y)$. Since $h$ is monotone, $y<x$. 

Hence a bound (i) for the number of such runs $R^{e,\eta}$ is given by $\sum_{y<x} M(e-1,y) $, and by  Claim~\ref{claim: local counting 2}, a bound for (ii) is $\max_{y<x}\max_{v<r(e-1,y)} B\left(x,q(v)\right)$.
\end{proof}

%%%%%%%%%%%%%%%%GR
We say that run of $R^{e, \eta}$ is a \emph{golden run} if $\eta\subset G$, the run is never cancelled, and every subprocedure $S^{e, \eta\aaa}_x$ with $\eta \aaa \subset G$ that is called by that run eventually returns or is cancelled. 

\begin{claim} \label{claim: GRGP 2} 
	If there is a golden run of $R^{e, \eta}$ with input (the final version of) $\PP^{e, \eta}$, then $\PP^{e, \eta},\Phi_e$ is a golden pair for $\Gamma$, $h$, and $G$, in the sense of Definition \ref{def:golden pair 2}, with parameter $\eta$.  
\end{claim}

\begin{proof}
Let $n$ be the parameter of this golden run. For almost all $x$ we have $h(x) > n$. To show the golden pair condition for such an $x$, suppose that $\Gamma^A(x)$ converges. Suppose $\eta \aaa \sub G$ where $|\aaa| = h(x)$. Choose $s_0$ so large that $\Gamma^A(x)$ and $\kuc ( \PP^{e,\eta}, \aaa)$ are stable from stage $s_0$ on, and (by Claim~\ref{claim: global counting 2} and the hypothesis that the run of $R^{e, \eta}$ is golden) all runs $S^{e, \eta \beta}_y$ for $y< x$ have returned or are cancelled. Following the instructions for $R^{e,\eta}$, we may now start a run $S^{e, \eta \aaa}_x$ (even if some $S^{e,\eta \nu}_y$ where $y> x $ and $\aaa \subseteq \nu$ is running and must be cancelled), and this run is not cancelled. Since this final run of $S^{e, \eta \aaa}_x$ returns, the $\PPI$ class in (\ref{eqn:defP}) becomes empty. Hence the golden pair condition for $x$ holds of $\PP^{e,\eta}, \Phi_e$. 
\end{proof}

\begin{claim} \label{clm:no-golden-run}
	Suppose that there is no golden run. Then for every $e$, there is some $\gamma\subset G$ such that there is a run of $R^{e,\gamma}$ that is never cancelled. 
\end{claim}

\begin{proof}
	By induction on $e$. For $e=0$ we have $\gamma=\estring$. Assume the lemma holds for $e$, witnessed by some $\gamma\subset G$. Since the run of $R^{e,\gamma}$ that is never cancelled is not golden, there is some $\aaa$ such that $\eta\aaa\subset G$ and such that there is a call of $S^{e,\eta\aaa}_x$ for some $x$ that is never cancelled but never returns. Then $S^{e,\eta\aaa}_x$ calls a run of $R^{e+1,\eta\aaa}$ that is never cancelled. 	
\end{proof}

%%%%%%%%%%%%%%%%%%%
It remains to show there is a golden run. For this we use the hypothesis that $A$ is computable from every $1$-random set $Z$ such that $G \ltt Z'$. We define the coding strings $\s_{\gamma}$ for $\gamma\in 2^{<\w}$. Let $\s_{\estring, s} = \estring$. 

\begin{enumerate}
	\item If $\s_{\eta, s}$ has been defined and procedure $R^{e, \eta}$ is running at stage~$s$ with input~$\PP$, then for all $\aaa \neq \estring$ such that no procedure $S^{e, \eta \beta}$ is running for any $\beta \subsetneq \aaa$, let $\s_{\eta \aaa,s}= \kuc(\PP, \aaa)$.

	\item If $\aaa$ is maximal under the prefix relation so that $\s_{\eta\aaa,s}$ is now defined, it must be the case that $S^{e,\eta\aaa}_x$ is currently running (for some $x$) and has called a run of $R^{e+1, \eta\aaa}$. This situation puts us back in case (1) with $\eta$ replaced by $\eta\aaa$, and the recursive definition can continue. 
\end{enumerate}

We verify that $\gamma \subset \delta$ implies $\s_{{\gamma}, s} \subset \s_{\delta, s}$ for each $s$ and $|\delta | \le s$. The fact to verify is that if some $R^{e,\eta}$ is running at stage $s$ with input $\PP^{e,\eta}$, and $S^{e,\eta\aaa}_x$ is also running at this stage and provides to $R^{e,\eta\aaa}$ the input $\PP^{e+1,\eta\aaa}$, then $\kuc(\PP^{e,\eta}_s,\aaa) \subseteq \kuc(\PP^{e+1,\eta\aaa}_s,\estring)$. The reason this fact holds is that we define $\PP^{e+1,\eta\aaa}$ as a subclass of $\PP^{e,\eta}\cap [\rho]$, where $\rho = \kuc(\PP^{e,\eta}_s,\aaa)$.  We may of course assume that this containment holds also for the stage $s$ clopen approximations to these classes, that is, that $\PP^{e+1,\eta\aaa}_s\subseteq \PP^{e,\eta}_s\cap [\rho]$. Now the desired extension $\rho \subseteq \kuc(\PP^{e+1,\eta\aaa}_s,\estring)$ holds because of measure considerations. Let $v$ be the $V$-index of $\PP^{e+1,\eta\aaa}$. Then by definition, $\kuc(\PP^{e+1,\eta\aaa},\estring)$ has length $q(v)$. From $\PP^{e+1,\eta\aaa}\subset [\rho]$ we conclude that $\leb \PP^{e+1,\eta\aaa}\le 2^{-|\s|}$. But we also know that $\leb \PP^{e+1,\eta\aaa}\ge 2^{-q(v)}$. Hence $q(v)\ge |\s|$. Since $\PP^{e+1,\eta\aaa}_s\subset [\rho]$, we conclude that indeed $\rho\subseteq \kuc(\PP^{e+1,\eta\aaa}_s,\estring)$.

\begin{claim}
For every $\gamma$, the number of stages $s$ such that $\s_{\gamma,s+1}\ne \s_{\gamma,s}$ is finite, and in fact is computably bounded in $\gamma$ (and hence in $|\gamma|$).  
\end{claim}

\begin{proof}
	We can have $\s_{\gamma,s+1} \ne \s_{\gamma,s}$ for two reasons:
\begin{enumerate}
	\item[(i)] A run of $S^{e, \eta}_x$ is called for some $\eta \sub \gamma$.
	\item[(ii)] A string $\kuc(\PP, \aaa)$ involved in the definition of $\s_{\gamma,s}$ changes from stage $s$ to stage $s+1$. 
\end{enumerate}	
	For any run of $S^{e,\eta}_x$ for $\eta\sub \gamma$ we must have $e,h(x)\le |\gamma|$, so a computable bound on the number of changes of type (i) is given by Claim~\ref{claim: global counting 2}(1). 
	
	If a string $\kuc(\PP,\aaa)$ is involved in the definition of $\s_{\gamma,s}$, then we must have $\PP=\PP^{e,\eta}$ for some $e\le |\gamma|$ and $\eta\aaa\subseteq \gamma$. Assuming that $e>0$, this run of $R^{e,\eta}$ was called by some $S^{e,\eta}_x$ where again $h(x)\le |\gamma|$. By Claim~\ref{claim: global counting 2}(2), $\PP = \PP^{(v)}$, where $v\le r(e-1,y)$ for some $y<x$. Thus, effectively in $\gamma$, we get a bound on $q(v)$ for the $V$-index $v$ of $\PP$, and so with the aid of Lemma \ref{lem:Kucera approx} (and again Claim~\ref{claim: global counting 2}(1)), a bound on the number of times a change as in (ii) may happen.  
\end{proof}

\

For all $\gamma$, let $\s_\gamma = \lim_s \s_{\gamma,s}$. Let $Z = \bigcup_{\gamma \prec G} \s_\gamma$. By a proof identical to the proof of Claim~\ref{clm:protosuperhigh-coding}, $G\ltt Z'$. By the assumption on $A$, we have $A\le_\Tur Z$. Hence $\Phi_e(Z)=A$ for some~$e$. 

Assume for a contradiction that there is no golden run. By Claim \ref{clm:no-golden-run}, there is a run of $R^{e,\eta}$, for some $\eta\subset G$, that is never cancelled. Since this run is not golden, there is some $\aaa$ such that $\eta\aaa\subset G$ and such that there is a run of $S^{e,\eta\aaa}_x$ for some $x$, called by $R^{e,\eta}$, that is never cancelled and never returns. This run defines 
\[ \PP^{e+1, \eta \aaa} = \left\{ X\in \PP^{e, \eta} \seq \aaa \,:\, \Phi_e(X) \nsupseteq \tau \right\},\]
where $\PP^{e,\eta}$ is the input of $R^{e,\eta}$. These classes are never altered, as $S^{e,\eta\aaa}_x$ is never cancelled, so we have $\tau\subset A$. Since the run of $S^{e,\eta\aaa}_x$ never return, the class $\PP^{e+1,\eta\aaa}$ is nonempty.

\begin{claim}
	$Z\in \PP^{e+1,\eta\aaa}$. 
\end{claim}

\begin{proof}
	Since $S^{e,\eta\aaa}_x$ is never cancelled, after its inception, no run $S^{e,\eta\beta}_y$ for any $\beta\subsetneq \aaa$ is ever called by $R^{e,\eta}$. Hence $\s_{\eta\aaa} = \kuc(\PP^{e,\eta},\aaa)$.
	
	Let $e'>e+1$. By Claim \ref{clm:no-golden-run}, there is some string $\gamma\subset G$ for which there is a run of $R^{e',\gamma}$ that is never cancelled. We have $|\gamma|\ge e'$. Then $\s_\gamma$ is extendible in $\PP^{e',\gamma}$ and $\PP^{e',\gamma}\subseteq \PP^{e+1,\eta\aaa}$. The result follows by compactness. 
\end{proof}

We now get the desired contradiction, since $\Phi_e(X)\ne A$ for all $X\in \PP^{e+1,\eta\aaa}$. This completes the proof of Proposition \ref{prop:golden pairs exist 2} and hence of Theorem~\ref{thm:main superhigh thm}. 
% \end{proof}

%%%%%%%%%%%%%%%%%%%%%SECTION 5.
\section{Demuth random sets and cost functions} \label{sec:last}

Consider the situation that  $A \leT Y$ where $A$ is c.e.\ and $Y$ is a $1$-random $\DII$ set. In this section we develop the connection between the strength  of cost functions $A$ can obey  and the degree of randomness of $Y$. This analysis will yield a proof of   Theorem \ref{thm:last}. 

We  gauge the degree of randomness of $Y$  via     the notion of Demuth randomness and its variants. 
 Recall from the first section the idea behind Demuth randomness. Tests are generalized in that one can change the  $m$-th component    for a computably bounded number of times.  We will introduce stronger randomness notions that are still compatible with being $\DII$ by relaxing the condition that the number of changes be computably bounded. Instead, each time there is a change to the current version of the $m$-th component,   we count down along a computable well-ordering $R$.

\subsection{$R$-approximations}  We begin with some intuitive background.
  Ershov~\cite{Ershov:68} introduced a  theory of $a$-c.e.\ functions  for an ordinal notation $a$ (see \cite{{Stephan.Yang.ea:nd}} for a recent survey).  The following  simpler variant     suffices for our purposes. We replace the ordinal notations by arbitrary  computable well-orders. 
Let $g(n,s)$ be a computable approximation to a function $f$. Suppose that the number of mind changes at $n$ is bounded by $h(n)$ where   $h$  is a  computable   function. We can think of the situation as follows: an ``approximator'' promises to give us a computable approximation to $f$. The approximator also has to give us evidence that the approximation will indeed stabilize on every input. Thus, for every $n$, at stage~$0$ the approximator puts a marker marked $n$ on the number $h(n)$ in the standard ordering $<$ of the natural numbers. Each time the approximator wants to change its approximation to $f(n)$, that is, at a stage $s$ at which $g(n,s)\ne g(n,s-1)$, the approximator needs to move the $n\tth$ marker at least one number to the left, that is, decrease its value in the ordering $<$. Since the ordering $<$ is a well-ordering of the natural numbers, this process ensures the stabilization of the approximation $g(n,s)$. The effectiveness of the entire setup is also due to the fact that $<$ is a computable well-ordering, and that the moves of the markers are given effectively. The notion of an $R$-approximation is identical, except that we replace $<$ by some other computable well-ordering of the natural numbers. 

For the rest of the section, we assume that every computable well-ordering $R$ we mention is infinite; indeed we assume that its field is $\w$.

\begin{definition} 
	Let $R = (\w,<_R)$  be a computable well-ordering. An \emph{$R$-ap\-prox\-i\-ma\-tion} is a computable function 
\begin{center}
$g = \la g_0, g_1 \ra: \w \times \w \rightarrow \w \times \w $ 
\end{center}
such that for each $x$ and each $s>0$, 
\begin{equation*}
%\label{eqn:Rapprox}
g(x,s) \neq g(x, s-1) \ria g_1(x, s) <_R g_1(x, s-1). 
\end{equation*}
In this case, $g_0$ is a computable approximation to a total $\Delta^0_2$ function $f$. We say that $g$ is an {$R$-approximation to $f$}. A $\DII$ function $f$ is called \emph{$R$-c.e.}~if it has an $R$-approximation. 
\end{definition}

\begin{lemma}\label{lem:w is R}
	Let $R$ be a computable well-ordering. Every $\w$-c.e.\ function is $R$-c.e.
\end{lemma}

\begin{proof}[Sketch of proof]
	If $S$ and $R$ are computable well-orderings, and $S$ is computably embeddable into $R$, then every $S$-c.e.\ set is $R$-c.e., because the effective embedding of $S$ into $R$ can be used to translate any $S$-approximation to an $R$-approximation. 
	
	Now, since we assume that $R$ is infinite, we can effectively embed $(\w,<)$ into $R$ by recursively choosing bigger and bigger elements in the sense of $<_R$. If the order-type of $R$ is not a limit ordinal, then we first fix a limit point of $R$, and then  build our embedding entirely below that limit point. 	
\end{proof}

The following lemma is related  to 
  Ershov's result that  each $\DII$ function is $a$-c.e.\ for some notation $a$ of $\omega^2$ (see \cite[Theorem 4.3]{Stephan.Yang.ea:nd}).   In our simpler setting,  a well-ordering of type $\omega$ suffices. 
\begin{lemma} \label{lem:Ershov} 
	For each computable approximation $g_0: \w \times \w \ria \w$ to a $\DII$ function~$f$, there is a computable well-ordering $R$ of order type $\w$ and a computable function $g_1: \w \times \w \ria \w $ such that $\la g_0, g_1 \ra$ is an $R$-approximation to $f$. 
\end{lemma}

\begin{proof}  It suffices to define a computable well-ordering $R$ with an infinite computable field $V \sub \w \times \w$.  Let  

\bc $V = \w \times \{0\} \cup \{ \la x, s \ra \colon \, s > 0 \lland g_0 (x,s)  \neq g_0(x, s-1) \}$. \ec
For $\la x,s \ra, \la y,t \ra \in  V$, we declare that $\la x, s \ra  <_R \la y, t \ra$ if $x< y$, or $x = y$ and $s > t$.  Then $R$ is   of order type $\w$  because $g_0$ is a computable approximation. Let $g_1(x, u)  = \la x, s \ra$ where $s \le u $ is largest such that $\la x,s \ra \in V$. Then $\la g_0, g_1 \ra$ is an $R$-approximation to $f$. 
\end{proof}

Later on we will need the following fact.
\begin{lemma}\label{lem:Rce universal function}
	For every  computable well-ordering $R$, there is a uniformly $\Halt$-com\-put\-able listing $\seq {f^e}$ of all $R$-c.e.\ functions. 
\end{lemma}

\begin{proof}
	Define a \emph{partial $R$-approximation} to be a partial computable function $\psi = \seq{\psi_0,\psi_1}\colon \w^2\to \w^2$ such that $\dom \psi$ is closed downward in both variables, and such that for all $n$ and $s>0$, if $(n,s)\in \dom \psi$ and $\psi(n,s)\ne \psi(n,s-1)$ then $\psi_1(n,s)<_R \psi_1(n,s-1)$. There is an effective listing $\seq{\psi^e}$ of all partial $R$-approximations. 
	
	 Write $\psi^e(n,t)\converge[s]$ to denote that $(n,t)\in \dom \psi^e$ and that this fact is discovered after $s$ steps of computation of some universal machine. We may assume that $\dom \psi^e\,[s]$ is closed downward in both variables. Given $e$, $n$, and $s$, let $t$ be greatest such that $\psi^e(n,t)\converge\![s]$, and let $g^e(n,s) = \psi_0(n,t)$. If there is no such $t$, then let $g^e(n,s)=0$.
	
	Now $\seq{g^e}$ is a uniformly computable sequence of functions, and the function $f^e$ defined by letting $f^e(n) = \lim_s g^e(n,s)$ is total for all $e$, so $\seq{f^e}$ is uniformly $\Halt$-computable. If $\psi^e$ is not total, then $\dom \psi^e$ is finite, whence $f^e(n)=0$ for almost all $n$. If $\psi^e$ is total then $f^e(n) = \lim_s \psi_0(n,s)$ for all $n$.
\end{proof}

\subsection{$R$-Demuth random sets}

\begin{definition} Let $R=(\w, <_R)$ be a computable well-ordering.
An $R$-\emph{Demuth test}  is a sequence of c.e.\ open sets $({\mathcal L}_m)\sN{m}$ such that $\fao m \leb {\mathcal L}_m \le \tp{-m}$, and there is a  function~$f$  with an $R$-approximation such that  ${\mathcal L}_m = \Opcl{W_{f(m)}}$.   A set~$Z$ \emph{passes} the test if $ Z\not \in   {\mathcal L}_m$ for almost all $m$.  We say that~$Z$ is  \emph{$R$-Demuth random}  if~$Z$ passes  each $R$-Demuth test.
	\end{definition}

Thus, a set is Demuth random if and only if it is $(\w,<)$-Demuth random. By Lemma \ref{lem:w is R},   for every computable well-ordering $R$, every $R$-Demuth random set is Demuth random.

\begin{proposition} \label{prop: ExDemRandom} 
	For every computable well-ordering $R$, there is a $\DII$ set $Y$ that is $R$-Demuth random. 
\end{proposition}

The proof of Proposition \ref{prop: ExDemRandom} is a variant of the construction of a $\DII$ Demuth random set; see \cite[Theorem 3.6.25]{Nbook}. The proof is divided into two parts. There is no universal $R$-Demuth test, but nevertheless, we first show that there is a special test $\seq{\GG_n}$ such that every set passing this test is $R$-Demuth random. Then we show that there is a $\Delta^0_2$ set that passes this special test.  In the following we write $\mathcal H_e $ for $ \Opcl{W_e}$.

\begin{definition}\label{def:special test}
	 A \emph{special test}  is a sequence of c.e.\ open sets $({\mathcal G}_m)\sN{m}$ such that  $\leb {\mathcal G}_m  \le \tp{-2m-1}$ for each~$m$  and there is a function $g \leT \Halt$ such that ${\mathcal G}_m = \mathcal H_{g(m)}$. $Z$  \emph{passes} the test if $  Z\not \in {\mathcal G}_m$ for almost all $m$.  \end{definition}

\begin{lemma}\label{lem:special_tests_force_Demuth}
	Let $R=(\w,<_R)$ be a computable well-ordering. There is a special test $\seq{\GG_n}$ such that every set that passes $\seq{\GG_n}$ is $R$-Demuth random. 
\end{lemma}

\begin{proof}
	A set   $Z$ is $R$-Demuth random  if{}f for each $R$-Demuth test $(\mathcal U_m)\sN{m}$, $Z$ passes the $R$-Demuth tests $(\mathcal U_{2m})\sN{m}$ and $(\mathcal U_{2m+1})\sN{m}$. Thus it  suffices to build a   special test $({\mathcal G}_m)\sN{m}$ that emulates all $R$-Demuth tests $(\mathcal S_n)\sN{n}$ such that $\leb \mathcal S_n \le \tp{-2n}$ for each~$n$.
 The idea is  now to put together all $R$-Demuth tests of this kind. This construction does not result in a universal $R$-Demuth test because the enumeration of all $R$-Demuth tests cannot be done effectively. However, it can be done effectively relative to  a $\Halt$ oracle, thus yielding a special test.

	The following definition will also be useful later. Given an effectively open class $\WW = \Opcl W $ (for some c.e.\ set $W$) and a positive rational number $\eps$, we can (uniformly in a c.e.\ index for $W$ and $\eps$) obtain a c.e.\ index for an effectively open class, which we denote by $\WW^{(\le \eps)}$, such that:
	\begin{enumerate}
		\item[(a)]  $\WW^{(\le \eps)} \subseteq \WW$;
		\item[(b)]  $\leb \WW^{(\le \eps)} \le \eps$; and
		\item[(c)]  If $\leb \WW \le \eps$, then $\WW^{(\le\eps)} = \WW$. 
	\end{enumerate}
The idea is simply to copy $W$, but prevent the enumeration of any string into $\WW^{(\le \eps)}$ that would make the measure go beyond $\eps$.

  By Lemma \ref{lem:Rce universal function} there is  a uniformly $\Halt$-computable listing  $\seq{f^e}$  of the $R$-c.e.\ functions. Thus, there is a function $\tilde q \leT \emptyset' $ such that   $\tilde q (e,n) = f^e(n)$ for each $e,n$.  Now let  $q \leT \Halt$ be a function such that 
  \bc $\HHH_{q(e,n)} = \HHH_{\tilde  q(e,n)}^{(\le \tp{-2n})}$ \ec 
and let 
\bc ${\mathcal G}_m = \bigcup_{e< m} \HHH_{q(e, e+m+1)}$. \ec
Then $\leb {\mathcal G}_m \le \sum_{e<m} \tp{-2(e+m+1)} \le \tp{-2m-1}$. If~$Z$ passes the special test $\seq{\GG_n}$ then it passes each $R$-Demuth test.  
 \end{proof}

The proof of Proposition \ref{prop: ExDemRandom} is completed with the following lemma.

\begin{lemma}\label{lem:special-tests-are-passed}
	If $\seq{\GG_n}$ is a special test, then   some $\DII$ set  $Z$  passes $\seq{\GG_n}$. 
\end{lemma}

\begin{proof} 	This lemma is Claim 2 of the proof of \cite[Theorem 3.6.25]{Nbook}. We give a sketch for completeness. For $n<\w$, let $\mathcal L_n = \bigcup_{m\le n} \GG_m$. 
	
	Recall (from Section \ref{sec:superhigh2}) that for a class $\WW$ and a string $\tau$, we let $\WW \mid \tau = \left\{ X\,:\, \tau X\in \WW\right\}$. If $\WW$ is measurable then $\leb(\WW \mid \tau) = 2^{|\tau|} \leb (\WW\cap [\tau])$. Note that $\leb(\WW \mid \tau)$ is the average of $\leb(\WW \mid \tau0)$ and $\leb(\WW \mid \tau1)$.
	
	With oracle $\Halt$, we recursively build a set $Z$ such that $$\leb(\mathcal L_n \mid Z\uhr n) \le 1- 2^{-n-1}$$ for all $n$. This inequality holds for $n=0$ since $\mathcal L_0 = \GG_0$ and $\leb \GG_0 \le 1/2$. If $Z\uhr n$ has been defined and the inequality holds, then $\Halt$ can determine a 1-bit extension $Z\uhr{n+1}$ of $Z\uhr n$ such that $\leb(\mathcal L_n \mid Z\uhr{n+1})\le 1- 2^{-n-1}$.
	
 Since  $\mathcal L_{n+1} = \mathcal L_n\cup \GG_{n+1}$ and $\leb(\GG_{n+1})\le 2^{-2n-3}$,  we have $\leb(\GG_{n+1} \mid Z\uhr{n+1})\le 2^{n+1}2^{-2n-3} = 2^{-n-2}$.
	 Thus  $\leb(\mathcal L_{n+1} \mid Z\uhr{n+1})\le 1-2^{-n-2}$.
	
  To finish we   show that  $Z\notin \GG_n$ for all $n$. If $Z\in \GG_n$ then $Z\in \mathcal L_n$, so there is some $m\ge n$ such that $[Z\uhr m]\subset \mathcal L_n$ 	since $\mathcal L_n$ is open;  but $\mathcal L_n\subseteq \mathcal L_m$, and $[Z\uhr m]\not\subseteq \mathcal L_m$. 
\end{proof}

Let $R^\w$ be the computable well-ordering of type ${|R|}^\w$ obtained from $R$ in the canonical way.  Suppose  $\seq{\GG_n}$ is the  special test obtained in Lemma~\ref{lem:special_tests_force_Demuth}.  Analyzing the proofs of the foregoing lemmas shows  that the   $R$-Demuth random set  $Z$ we build is  $R^\w$-c.e.

\subsection{$R$-benignity}

Given a monotone cost function~$c$, we define a computable function $  q,s \mapsto w_c(q,s)$    where $q$ is a non-negative rational   and   $s \in \w$. Let $w_c(q, 0)=0$. For $s>0$ let 
\begin{equation} \label{eqn:www} w_c(q, s)  =  \begin{cases} s  & \text{if} \ c(w_c(q, s-1), s) \ge q \\ w_c(q,s-1) & \text{otherwise}. \end{cases} \end{equation} 

We  now count the number of times $w_c(q,s)$ changes. Clearly, $c$ satisfies the limit condition  $\lim_x \sup_s c(x,s)=0$  if and only if $w_c(q,s)$ changes only finitely often for each $q$, and $c$ is benign in the sense of Definition~\ref{df:benign} if and only if this number of changes is in fact computably bounded in $q$.  We use $R$-approximations to get a hold on  cost functions that satisfy the limit condition but are not necessarily benign.

\begin{definition}\label{def:R-benign}
	Let $R$ be a computable well-ordering. A monotone cost function~$c$ is \emph{$R$-benign} if there is a computable function $f\colon \w^2\to \w$ such that $\seq{w_c,f}$ is an $R$-approximation. 
\end{definition}

That is, we not only require that the function  $q \mapsto \lim_s w_c(q,s)$ be $R$-c.e., we actually require that the canonical approximation $w_c(q,s)$ be extendible to an $R$-approximation. By the monotonicity of $c$, it is sufficient to replace all rational numbers $q$ by a computable sequence of rational numbers $\seq{q_n}$ descending to 0. Clearly, a cost function is benign if and only if it is $(\w,<)$-benign.  From Lemma~\ref{lem:Ershov}  we have  that a monotone cost function $c$ satisfies the limit condition if and only if there is some computable well-ordering $R$ such that $c$ is $R$-benign. 

%Clearly the limit condition $\lim_x \sup_s c(x,s)=0$ is equivalent to the existence of $\lim_s w_c(q,s)$ for each $q$. 
 
\begin{remark}\label{rem:plus n}
Suppose that the computable well-ordering $R$ is of order type $\alpha + n$, where $n$ is finite, and let $S$ be a computable well-ordering of type $\alpha$ obtained from $R$ in a natural way. Let $c$ be an $R$-benign monotone cost function, as witnessed by the $R$-approximation $\seq{w_c,f}$. Suppose that for each $\epsilon>0$, there are a non-negative rational $q<\epsilon$ and an $s$ such that
$f(q,s)$ is in the $n$ part of $R$. Then it is not hard to check that $c$ is in fact benign. Otherwise, we can adjust $f$ to obtain an $S$-approximation. Thus, every $R$-benign monotone cost function is in fact $S$-benign.

So if $R$ is a computable well-ordering such that there are $R$-benign cost functions that are not $S$-benign for any computable proper initial segment of $R$, then $R$ has order type $\omega \cdot \alpha$ for some $\alpha$.
\end{remark}

\subsection{Main result of this section}

 Recall that the product $A\cdot B$ of two linear orderings $A$ and $B$ is the linear ordering obtained by replacing every point in $B$ by a copy of $A$. In other words, it is the right-lexicographic ordering on $A\times B$.

\begin{theorem} \label{thm:cRbenign} 
	Let $R$ be a computable well-ordering, and let $c$ be an $R$-benign cost function. Let $Y$ be any $\w\cdot R$-Demuth random set. Then every c.e.\ set  $A$ that is computable from $Y$ obeys $c$. 
\end{theorem}

Before proving Theorem \ref{thm:cRbenign}, we show how it implies Theorem \ref{thm:last}.

\begin{proof}[Proof of Theorem \ref{thm:last}]
	Let $c$ be a monotone cost function that satisfies the limit condition. By Lemma~\ref{lem:Ershov}, there is some computable well-ordering $R$ such that $c$ is $R$-benign. By Proposition \ref{prop: ExDemRandom}, there is some $\DII$ set $Y$ that is $\w\cdot R$-Demuth random. By Theorem \ref{thm:cRbenign}, every c.e.\ set computable from $Y$ obeys $c$. 
\end{proof}

\begin{proof}[Proof of Theorem \ref{thm:cRbenign}] Every  Demuth random set is GL$_1$ (see \cite[Thm.\ 3.6.26]{Nbook}) and hence Turing incomplete,  so  $A$ is a base for $1$-randomness by Fact~\ref{fact:ce-base-for-randomness}. It follows that  $A$ is low for $K$ and therefore  superlow (see \cite[Cor.\  5.1.23 and Prop.\ 5.1.3]{Nbook}).

  We define the numbers $w_c(q,s)$ by (\ref{eqn:www}) and  use the shorthand  \bc  $v(m, s)=w_c(\tp{-2m},s)$. \ec Note that $v(m) = \lim_s v(m,s)$ exists for each $m$ by the hypothesis on $c$. Furthermore, we may assume the function  $\lambda m. v(m)$ is unbounded; otherwise, every c.e.\ set   obeys~$c$. 
  
  Let $\Gamma$ be a Turing functional such that, if $ v=v(m,s)$ is the $i$-th value of the parameter $v(m,\cdot)$, then $\Gamma^X(m,i)$ is defined with use $v$. Since the c.e.\ set $A$ is superlow, $A$ is jump-traceable by \cite{Nies:Little}. So  it is easy to see      that  there is  a computable function $p$ and there is a computable enumeration $\seq {A_s}\sN s$ of $A$ such that  $\Gamma^{A_s}(m,i)$ is  destroyed at most $p(m,i)$ many times.  For a detailed proof  of this simpler variant of Theorem~\ref{thm: JT and SL}   see \cite[Lemma 3.1]{Kucera.Nies:ta}.

Let $\Phi$ be a Turing functional such that $A= \Phi^Y$ for some $\w \cdot R$-Demuth random  set~$Y$.  To ensure that $A$ obeys $c$ we want to restrict the changes  of $A\uhr{v(m,s)}$. To do so we define an $\w \cdot R$-Demuth test based on the following:   let ${\mathcal L}_m[s]$ go through all the   c.e.\ open  sets
\begin{center} $\{Z \colon \, A_s \uhr{v(m,s)} \subseteq \Phi^Z\} $. \end{center}
That is, as long as $A_s \uhr{v(m,s)}$ remains unchanged we have the same version of ${\mathcal L}_m$. 
 The idea is that the measure of the final version of  ${\mathcal L}_m$ has to exceed $\tp{-m}$ for almost all $m$, otherwise $Y$ would fail the $\w \cdot R$-Demuth test obtained by stopping the enumeration of ${\mathcal L}_m$ when its measure attempts to exceed that bound. 
If $A\uhr{v(m,s)}$ changes then we start a new version of ${\mathcal L}_m$. This type of change can happen at most $p(m,i)$ many times  while  the parameter $v(m,\cdot)$ has its  $i$-th value.

For the formal details,  fix  a computable function $h_0$ such that    $\Opcl{ W_{ h_0(m,s)}}= {\mathcal L}_m[s]$  for each $s$.

%Let ${\mathcal L}_m$ denote the final version of ${\mathcal L}_m[s]$.

\begin{claimin}One can  extend $h_0$ to an  $\w \cdot R$-approximation $\la h_0, h_1\ra$.\end{claimin}

%\emph{$\seq{G_m}\sN  m$ is an $\w \cdot R$-Demuth test.}

\n By the hypothesis that  $c$ is $R$-benign, there is a computable function $f\colon \w \times \w \ria \w$ such that $\la v,f \ra$ is an $R$-approximation. The idea is now
to follow this $R$-approximation if  $v(m,s)$ changes, and use the first components of the pairs in $\w \cdot R$ to count the changes of $A \uhr{v(m,s)}$ while $v(m,s)$ is stable. For the formal details,  we define the computable function $h_1$ by 

\bc $h_1(m,s) = \la l,  f(m,s)\ra$,  \ec
where the counter $l \in \w$ is  initialized at $p(m,i)$ when $v(m,s)$ assumes its $i$-th value (recall that $p$ is the function such that $\Gamma^{A_s}(m,i)$ is  destroyed at most $p(m,i)$ many times). Subsequently, each time $A \uhr {v(m,s)}$ changes while $v(m,s)$ remains at this value, we decrease this counter. This completes the proof of Claim~1.

Recall from the proof of Proposition~\ref{prop: ExDemRandom} that  for a c.e.\ open set  $\WW$ and a positive rational $\eps$, we let $\WW^{(\le \eps)}$ denote  a uniformly obtained c.e.\ open set contained in $\WW$  that equals $\WW$ if the measure of $\WW$  does not exceed $\eps$. 

 Let $\mathcal H_m[s] = ({\mathcal L}_m[s])^{(\le \tp{-m})}$. Let $\mathcal H_m$ denote the final version of $\mathcal H_m[s]$. Then, by Claim~1, $\seq {\mathcal H_m}\sN m$
is  an $\w \cdot R$-Demuth test.

 By the hypothesis on $Y$, there is $m^*$ such that $Y \not \in \mathcal H_m$ for each $m \ge m^*$. Then, since $A \uhr {v(m)} \sub \Phi^Y$, we have  $\leb {\mathcal L}_m  >  \tp{-m}$ for each $m \ge m^*$: otherwise~$Y$ would enter $\mathcal H_m$. 

It remains to  obtain a computable  enumeration $\seq{\hat A_s}\sN s$ of $A$ obeying $c$. We first define an infinite  computable sequence of stages: let $s_0= m^*$,  and   
\begin{center} $s_{i+1} = \mu s > s_i\,  \fa m \, [ m^* \le m \le s_i \, \ria  \leb \mathcal L_{m,s} [s] > \tp{-m}]$. \end{center}
Let 

\bc $g(i) = \max\{ v(m,s_i)\colon \, m^*\le m\le  s_i\}$.  \ec 

Now, consider a stage $s$ such that    $s_i  \le s < s_{i+1}$. Given $x$,   since  the function  $\lambda m. v(m) $ is unbounded, there is a least  positive $j\ge i$ such that $g(j-1) > x$.  Let

\begin{center} $\hat A_s(x) = A_{s_{j+2}}(x) $. \end{center}

\begin{claimin}The computable enumeration $\seq{\hat A_s}\sN s$ obeys $c$.\end{claimin}

\n We have to show that the total cost of changes for this enumeration, as defined in (\ref{eqn:total sum}) in Subsection~\ref{subsec:cost-functions},    is finite. Suppose that  at a stage $s$, the number $x$ is least such that $\hat A_{s-1}(x) \neq \hat A_s(x)$.  Then $s=s_j$ for some~$j$ such that $g(j-1) > x$. So, we can choose a least $m_s= m < s_j$ such that $x< v(m,s)$. We may assume that $x > v(m^*+1)$.  Then  $m>m^*$ and  $v(m-1, s) \le x $. Recall that  $v(m-1,s)= w_c(\tp{-2m+2})$. Then, by definition  and the monotonicity of $c$ we have $c(x,s) <  \tp{-2m+2}$. 

Since  $g(j-1) > x$, we have $  A_{s_{j+1}}(x) \neq A_{s_{j+2}}(x)$, so  all the versions  ${\mathcal L}_m[s]$ for $s > s_{j+1}$ are disjoint from ${\mathcal L}_m[s_{j+1}]$. Then, since $\leb {\mathcal L}_m[s_{j+1}]> \tp{-m}$, a situation as  above for a particular value $m$ can occur at most $\tp{m}$ many times. (That is,  there are at most $2^m$ stages $s$ such that $m_s= m$.) Thus, the total cost of changes at numbers $x > v(m^*+1)$  for this computable enumeration of $A$ is bounded by $\sum_m \tp{m}\tp{-2m+2} = 8$. 
\end{proof}

To end the paper, we remark that if   $R = (\w, <)$, we can actually  obtain an $R$-approximation in   Claim 1, since   the current version $\mathcal L_m[s]$ of $\mathcal L_m$ changes at most $  p(m,f(m,0) ) \cdot f(m,0) $ many times. Thus,  the test $\seq{\mathcal H_m}\sN  m$ is  an  $(\w, <)$ Demuth test, i.e.,   a Demuth test in the usual sense.  It follows that   each  c.e.\ set $A$ Turing  below a Demuth random set obeys every  benign cost function. As mentioned in the first section,  \Kuc\ and Nies  \cite{Kucera.Nies:ta} had previously obtained  the  equivalent result    that such a set~$A$ is strongly jump-traceable.

The above argument works equally well if $R$ is of order type $\omega \cdot \alpha$ for some ordinal $\alpha$ and the set of limit points of $R$ is computable. Thus, in this case Theorem \ref{thm:cRbenign} can be strengthened by weakening the hypothesis on $Y$ from $\omega \cdot R$-Demuth randomness to $R$-Demuth randomness. This fact is particularly interesting given Remark \ref{rem:plus n}.

By further adapting  the   techniques above  to plain  Demuth tests, we obtain a new proof of the harder  right-to-left direction of Characterization~Ia  for c.e.\ sets $A$: if $A$ is below each $\w$-c.e.\ $1$-random set then $A$ obeys each benign cost function~$c$ (and hence $A$  is strongly jump-traceable).

Firstly, given $A$ and $c$,  define  the computable  functions $p$ and $f$ as above.  Let $r(m) = p(m,f(m,0) ) \cdot f(m,0) $.  The proof of Lemma~\ref{lem:special_tests_force_Demuth} shows that there is a Demuth test $\seq{\mathcal G_n}\sN  n$ (taking the role of the special test there) such that each set passing  $\seq{\mathcal G_n}\sN  n$   passes each  Demuth test $\seq{\mathcal   H_m} \sN  m$ with at most  $r(m) $  many changes to the current  version of the $m$-th component.  
Secondly,  since $\seq{\mathcal G_n}\sN  n$ is now a Demuth test, by the proof of Lemma \ref{lem:special-tests-are-passed} there is an $\w$-c.e.\ set $Y$ passing it. 
 Finally, by hypothesis,  $A = \Phi(Y)$ for some Turing functional $\Phi$. Define  the Demuth test  $\seq{\mathcal H_m}\sN  m$ as  in the proof of Theorem~\ref{thm:cRbenign}. Then the number of times a version $\mathcal H_m[s]$ changes is bounded by $r(m) $, so $Y$ passes this Demuth test.  Now  the proof of Theorem~\ref{thm:cRbenign}  shows  that $A$ obeys~$c$.

% 

%\bibliographystyle{plain}
%\bibliography{../refs_and_macros/andrepapers,../refs_and_macros/random,../refs_and_macros/various,../refs_and_macros/peter}

%

%

%

\end{document}